\documentclass[reqno]{amsart}
\usepackage{amsmath, mathtools, mathrsfs}
\usepackage{amsfonts}
\usepackage{amssymb}
\usepackage{graphicx,color}
\usepackage[square,numbers]{natbib}
\usepackage{verbatim}
\newtheorem{theorem}{Theorem}
\newtheorem{claim}[theorem]{Claim}

\newtheorem{lemma}[theorem]{Lemma}
\newtheorem{proposition}[theorem]{Proposition}

\theoremstyle{definition}
\newtheorem{definition}[theorem]{Definition}

\newtheorem{notation}[theorem]{Notation}

\theoremstyle{remark}
\newtheorem{remark}[theorem]{Remark}

\newtheorem{case}{Case}

\numberwithin{theorem}{section}
\numberwithin{equation}{section}

\def\XXint#1#2#3{{\setbox0=\hbox{$#1{#2#3}{\int}$}
     \vcenter{\hbox{$#2#3$}}\kern-.5\wd0}}

\newcommand{\dd}{\; \mathrm{d}}

\newcommand{\bbH}{\mathbb{H}}

\newcommand{\bbE}{\mathbb{E}}
\newcommand{\bbG}{\mathbb{G}}

\newcommand{\bbN}{\mathbb{N}}

\newcommand{\bbR}{\mathbb{R}}
\newcommand{\bbQ}{\mathbb{Q}}

\begin{document}
\title[UDS and Maximal Directional Derivatives in Carnot Groups]{Universal Differentiability Sets and Maximal Directional Derivatives in Carnot Groups}

\author[Enrico Le Donne]{Enrico Le Donne}
\address[Enrico Le Donne]{Department of Mathematics and Statistics, University of Jyvaskyla, 40014 Jyvaskyla, Finland}
\email[Enrico Le Donne]{Enrico.E.LeDonne@jyu.fi}

\author[Andrea Pinamonti]{Andrea Pinamonti}
\address[Andrea Pinamonti]{Department of Mathematics, University of Trento, Via Sommarive 14, 38050 Povo (Trento), Italy}
\email[Andrea Pinamonti]{andrea.pinamonti@unitn.it}

\author[Gareth Speight]{Gareth Speight}
\address[Gareth Speight]{Department of Mathematical Sciences, University of Cincinnati, 2815 Commons Way, Cincinnati, OH 45221, United States}
\email[Gareth Speight]{Gareth.Speight@uc.edu}

\keywords{Carnot group, Lipschitz map, Pansu differentiable, directional derivative, universal differentiability set}

\date{\today}

\begin{abstract}
We show that every Carnot group $\bbG$ of step 2 admits a Hausdorff dimension one `universal differentiability set' $N$ such that every Lipschitz map $f\colon \bbG \to \bbR$ is Pansu differentiable at some point of $N$. This relies on the fact that existence of a maximal directional derivative of $f$ at a point $x$ implies Pansu differentiability at the same point $x$. We show that such an implication holds in Carnot groups of step 2 but fails in the Engel group which has step 3.
\end{abstract}

\maketitle

\section{Introduction}\label{intro}
Rademacher's theorem asserts that every Lipschitz function $f\colon \bbR^{n}\to \bbR^{m}$  is differentiable almost everywhere with respect to Lebesgue measure. One direction of research has been to study Rademacher's theorem for more general measures \cite{AM14, DhR} and in more general spaces, including Carnot groups \cite{Pan89}, Banach spaces \cite{LPT13}, and metric measure spaces \cite{Che99, Bat15}. 

Another direction of research asks whether Rademacher's theorem admits a converse: given a Lebesgue null set $N\subset \bbR^{n}$, does there exist a Lipschitz map $f\colon \bbR^{n}\to \bbR^{m}$ which is differentiable at no point of $N$? The answer is yes if and only if $n\leq m$. The solution is easy if $n=m=1$ \cite{Zah46}, while the cases $n>m$ are covered in \cite{Pre90, PS15}, and the cases $n\leq m$ in \cite{ACP10, CJ15}. In the case $n>m=1$, \cite{DM11, DM12, DM14} strengthened the result of \cite{Pre90}, showing that $\bbR^{n}$ contains a compact Hausdorff (even Minkowski) dimension one set containing a point of differentiability for every Lipschitz map $f\colon \bbR^{n}\to \bbR$. Sets containing a point of differentiability for all real-valued Lipschitz functions were called \emph{universal differentiability sets}.

The aim of this paper is to combine these directions of research and investigate universal differentiability sets (UDS) in Carnot groups (Definition \ref{Carnot}). A Carnot group is a Lie group whose Lie algebra admits a stratification. This decomposes the Lie algebra as a direct sum of vector subspaces, the first of which generates the other subspaces via Lie brackets. The number of subspaces is called the step of the Carnot group and to some extent indicates its complexity. Carnot groups have a rich geometric structure, including translations, dilations, Carnot-Carath\'{e}odory (CC) distance, and a Haar measure \cite{CDPT07, Gro96, LeD, Mon02, Spe14, Vit14}.

By replacing Euclidean translations and dilations with translations and dilations in the Carnot group, one can define Pansu differentiability of functions between Carnot groups (Definition \ref{pansudifferentiability}). Pansu generalized Rademacher's theorem to Carnot groups, showing that Lipschitz functions between Carnot groups are Pansu differentiable almost everywhere with respect to the Haar measure (Theorem \ref{pansutheorem}) \cite{Pan89, MPS17}. This can be applied to show that every Carnot group (other than Euclidean space itself) contains no subset of positive measure which bi-Lipschitz embeds into a Euclidean space \cite{HajMal, Mag00, Sem}.

The study of UDS in Carnot groups began in \cite{PS16}, where the second and third authors showed that the Heisenberg group contains a measure zero UDS. Our main theorem extends this result to more general Carnot groups. By CC-Hausdorff dimension of a set $N\subset \bbG_{r}$ we will mean the Hausdorff dimension of $N$ with respect to the Carnot-Carath\'eodory distance.

\begin{theorem}\label{maintheorem}
Let $\mathbb{G}$ be a step 2 Carnot group. Then there exists a set $N\subset \mathbb{G}$ of CC-Hausdorff dimension one such that every Lipschitz function $f\colon \mathbb{G} \to \mathbb{R}$ is Pansu differentiable at a point of $N$.
\end{theorem}

Techniques to construct measure zero UDS originate from the work of Preiss \cite{Pre90}. The main result of \cite{Pre90} was that if the norm of a Banach space $E$ is Frechet differentiable away from the origin, then every Lipschitz function is Frechet differentiable in a dense set of points. This relied upon the idea that if $\|e\|=1$ and $f'(x,e)=\mathrm{Lip}(f)$, then $f$ is Frechet differentiable at $x$ \cite{Fit84}. Since an arbitrary Lipschitz function need not admit such a maximal directional derivative, Preiss introduced a notion of `almost maximal' directional derivatives. He showed these exist for every Lipschitz function (possibly after a linear perturbation) and suffice for differentiability. Since only directional derivatives are involved, this procedure can be carried out inside a measure zero subset of $\bbR^{n}$ containing sufficiently many lines, giving a measure zero UDS in $\bbR^{n}$ as a corollary.

After recalling the necessary background in Section \ref{preliminaries}, we begin by investigating maximal directional derivatives in Carnot groups in Section \ref{strongmaximal}. Fix an inner product norm $\omega$ on the horizontal layer $V_{1}$ of a Carnot group $\bbG$, with corresponding CC distance $d$ (Definition \ref{carnotdistance}). Using the notation $d(x)=d(x,0)$, our replacement for differentiability of the norm in Banach spaces is as follows.

\begin{definition}\label{defdifferentiabilityofdistance}
We say that the CC distance $d$ in $\bbG$ is \emph{differentiable in horizontal directions} if it is Pansu differentiable at every point $u$ of the form $u=\exp E$ for some $E\in V_{1}\setminus \{0\}$.
\end{definition}

The directional derivative $Ef(x)$ of a Lipschitz function $f\colon \bbG \to \bbR$ at a point $x\in \bbG$ in the direction of $E\in V_{1}$ is defined in the natural way, by differentiating the composition of $f$ with the flow $t\mapsto x\exp tE$ (Definition \ref{defdirectionalderivative}). The relationship between directional derivatives and Pansu differentiability was investigated using porosity in \cite{PS16pansu}. It is not hard to show that $|Ef(x)|\leq \omega(E)\mathrm{Lip}(f)$. Combining the results of Theorem \ref{maximalimpliesdiff} and Proposition \ref{Ed=1} gives the following equivalence result.

\begin{proposition}\label{equivalence}
In a Carnot group $\bbG$, the CC distance is differentiable in horizontal directions if and only if the following implication holds: whenever $f\colon \bbG \to \bbR$ is Lipschitz and there exist $x\in \bbG$ and $E\in V_{1}$ with $\omega(E)=1$ and $Ef(x)=\mathrm{Lip}_{\bbG}(f)$, then $f$ is differentiable at $x$.
\end{proposition}

In \cite{PS16} it was shown that Heisenberg groups enjoy the properties stated in Proposition \ref{equivalence}. We generalize this positive result to all step 2 Carnot groups. We do this first in free Carnot groups of step 2 by explicit construction of horizontal curves and estimates of the CC distance (Lemma \ref{curve} and Theorem \ref{thmdifferentiabilityofdistance}). We show later in Section \ref{morecarnotgroups} that differentiability of the CC distance is inherited under suitable images of Carnot groups, allowing us to pass to arbitrary step 2 Carnot groups (Proposition \ref{quotientdiffCC}). 

In Section \ref{engelsection} we show that differentiability of the CC distance fails in the Engel group, a step 3 Carnot group (Theorem \ref{Conterex:Engel}). Combining results gives the following.

\begin{theorem}\label{whatgroupsdiff}
The CC distance is differentiable in horizontal directions (or, equivalently, maximality implies differentiability) in any step 2 Carnot group, but not in the Engel group (a step 3 Carnot group).
\end{theorem}

In Section \ref{sectionUDS} we turn to construction of the UDS. Since known techniques rely on construction of `almost maximal' directional derivatives, Theorem \ref{whatgroupsdiff} forces us to restrict ourselves to step 2 Carnot groups. We show that in such groups one can carry out the procedure first used in the linear setting by Preiss \cite{Pre90} and adapted to the Heisenberg group in \cite{PS16}. We work initially in free Carnot groups of step 2. Since the coordinate representation of such groups is a clear generalization of that of Heisenberg groups, we are able to avoid repeating much of the work from \cite{PS16}. We prove again only those results where the geometry of the individual Carnot group is important (Lemma \ref{closedirectioncloseposition} for estimating distances and Lemma \ref{curveforalmostmax} for construction of curves). The UDS $N$ is then a suitably chosen $G_{\delta}$ set containing images of curves from Lemma \ref{curveforalmostmax} with rational parameters (Lemma \ref{uds}). By a simple argument (as found in \cite{DM12}), $N$ can be made not only measure zero but also of CC-Hausdorff dimension one. Existence of `almost maximal' directional derivatives at points of $N$ (Proposition \ref{DoreMaleva}) and their sufficiency for differentiability (Proposition \ref{almostmaximalityimpliesdifferentiability}) then follow exactly as in \cite{PS16}. This gives Theorem \ref{maintheorem} for free Carnot groups of step 2.

Finally, in Section \ref{morecarnotgroups} we see how to pass from free Carnot groups of step 2 to general step 2 Carnot groups. Suppose $F\colon \bbG \to \bbH$ is a Lie group homomorphism between Carnot groups preserving the first layer of the Lie algebras. We show that if the CC distance in $\bbG$ is differentiable then the CC distance in $\bbH$ is differentiable (Proposition \ref{quotientdiffCC}). Further, if $N\subset \bbG$ is a CC-Hausdorff dimension one UDS in $\bbG$, then $F(N)\subset \bbH$ is a CC-Hausdorff dimension one UDS in $\bbH$ (Proposition \ref{quotientUDS}). Since any step 2 Carnot group is the image of a free Carnot group of step 2 under such a map, this proves Theorem \ref{maintheorem} for general step 2 Carnot groups.

We now mention several directions in which this work could be extended.

Firstly, one could try to generalize the techniques of \cite{DM11, DM12, DM14} to construct a compact Hausdorff/Minkowski dimension one UDS in step 2 Carnot groups. Since the present techniques are already complicated we decided not to do so, focusing instead on how the geometry of Carnot groups comes into play. Note that there is a limit to how small a UDS can be, e.g. a UDS cannot be $\sigma$-porous \cite{PS16structure}.

In light of Theorem \ref{whatgroupsdiff}, a natural question is whether the CC distance can be differentiable in horizontal directions only in step 2 groups, or there exist higher step groups enjoying this property. At present we do not know the answer. The reason is that not all 3-step Carnot groups have the Engel group as a quotient \cite{Marchi14}.

Finally, one may ask whether any higher step Carnot groups may admit measure zero UDS. If the CC distance is differentiable in some higher step groups then it may be possible to extend existing techniques to those. However, for the Engel group one would need different methods entirely.

\medskip

\noindent \textbf{Acknowledgement.} E.L.D.~is supported by the Academy of Finland grant 288501 and by the ERC Starting Grant 713998 GeoMeG. A.P. acknowledges the support of the Istituto Nazionale di Alta  Matematica F. Severi. G.S. received support from the Charles Phelps Taft Research Center.

\section{Preliminaries}\label{preliminaries}
In this section we recall relevant information for general Carnot groups.

\subsection{General Carnot groups}
We recall that a Lie algebra is a vector space $V$ equipped with a Lie bracket $[\cdot,\cdot] \colon V\times V\to V$ that is bilinear and satisfies $[x,x]=0$ and the Jacobi identity $[x,[y,z]]+[z,[x,y]]+[y,[z,x]]=0$.

\begin{definition}\label{Carnot}
A \emph{Carnot group} $\bbG$ of \emph{step} $s$ is a simply connected Lie group whose Lie algebra $\mathfrak{g}$ admits a decomposition as a direct sum of subspaces of the form
\[\mathfrak{g}=V_{1}\oplus V_{2}\oplus \cdots \oplus V_{s}\]
such that $V_{i}=[V_{1},V_{i-1}]$ for any $i=2, \ldots, s$, and $[V_{1},V_{s}]=0$. The subspace $V_{1}$ is called the \emph{horizontal layer} and its elements are called \emph{horizontal left invariant vector fields}. The \emph{rank} of $\bbG$ is $\dim V_{1}$.
\end{definition}

The exponential mapping $\exp\colon \mathfrak{g}\to \bbG$ is a diffeomorphism. Given a basis $X_{1},\ldots, X_{n}$ of $\mathfrak{g}$ adapted to the stratification, any $x\in \bbG$ can be written in a unique way as
\[x=\exp(x_{1}X_{1}+\ldots +x_{n}X_{n}).\]
We identify $x$ with $(x_{1},\ldots, x_{n})\in \bbR^{n}$ and $\bbG$ with $(\bbR^{n},\cdot)$, where the group operation on $\bbR^{n}$ is determined by the Baker-Campbell-Hausdorff formula on $\mathfrak{g}$. This is known as \emph{exponential coordinates of the first kind}.

A curve $\gamma \colon [a,b]\to \bbG$ is absolutely continuous if it is absolutely continuous as a curve into $\bbR^{n}$.

\begin{definition}\label{horizontalcurve}
Fix a basis $X_{1}, \ldots, X_{r}$ of $V_{1}$, which are seen as left invariant vector fields. An absolutely continuous curve $\gamma\colon [a,b]\to \bbG$ is \emph{horizontal} if there exist $u_{1}, \ldots, u_{r}\in L^{1}[a,b]$ such that
\[\gamma'(t)=\sum_{j=1}^{r}u_{j}X_{j}(\gamma(t))\]
for almost every $t\in [a,b]$. The \emph{length} of such a curve is
\[L_{\bbG}(\gamma)=\int_{a}^{b}|u|.\]
\end{definition}

Since $\bbG$ is identified with $\bbR^{n}$ as a manifold, all its tangent spaces can be naturally identified with $\bbR^{n}$. We say that a vector $v\in \bbR^n$ is \emph{horizontal} at $p\in \bbG$ if $v=E(p)$ for some $E\in V_1$. Thus a curve $\gamma$ is horizontal if and only if $\gamma'(t)$ is horizontal at $\gamma(t)$ for almost every $t$. Chow's theorem \cite[Theorem 19.1.3]{BLU07} asserts that any two points in a Carnot group can be connected by a horizontal curve. Hence the following definition is well-posed.

\begin{definition}\label{carnotdistance}
For every $x,y\in \bbG$, their \emph{Carnot-Carath\'{e}odory (CC) distance} is defined by
\[d(x,y)=\inf \{L_{\bbG}(\gamma)\colon \gamma \mbox{ is a horizontal curve joining } x \mbox{ to }y\}.\]
We also use the notation $d(x)=d(x,0)$ for $x\in \bbG$.
\end{definition}

\begin{remark}\label{horizcurve}
In every Carnot group (actually in every subRiemannian group), the CC
distance can be equivalently defined by taking the infimum of the length of
piecewise linear horizontal curves.
Namely, if $\hat d(x,y)$ is the  infimum of lengths of Lipschitz
horizontal curves joining $x$ to $y$ which are the concatenation of
finitely many straight lines, then we claim that $\hat d = d$.
Indeed, $\hat d $ is defined via a length structure (according the
terminology in \cite{Burago-Burago-Ivanov}), hence it is a length
distance. Notice that $\hat d $ is finite and induces the manifold topology because the
proof of Chow's theorem on Lie groups gives exactly a piecewise linear
horizontal curve joining each pair of points, see \cite{Gro96, LeD}. Moreover, $\hat d $ is left-invariant by construction.
By a theorem of Berestowski \cite{Ber}, we conclude that $\hat d $ is a
subFinsler metric. Now, since both $d $ and $\hat d $ admit the same dilations, they have
the same horizontal bundle; since they give the same length to
horizontal lines (which in both cases are geodesics), the two
distances have the same norm for horizontal vectors. We conclude that
$d $ and $\hat d $ coincide.
\end{remark}

Left group translations preserve lengths of horizontal curves. This implies that $d(gx,gy)=d(x,y)$ for every $g,x,y \in \bbG$.

\begin{definition}\label{dilations}
\emph{Dilations} $\delta_{\lambda}\colon \bbG \to \bbG$, $\lambda>0$, are defined on $\bbG$ in coordinates by
\[\delta_{\lambda}(x_{1}, \ldots, x_{n})=(\lambda^{\alpha_{1}}x_{1},\ldots, \lambda^{\alpha_{n}}x_{n})\]
where $\alpha_{i}\in \bbN$ is the homogeneity of the variable $x_{i}$. For our purposes, it will be enough to know that $\alpha_{1}=\cdots=\alpha_{r}=1$,  where $r=\dim V_{1}$.
\end{definition}

Dilations are group homomorphisms of $\bbG$ and satisfy $d(\delta_{\lambda}(x),\delta_{\lambda}(y))=\lambda d(x,y)$ for every $x, y \in \bbG$ and $\lambda>0$. We will also use the fact that $\delta_{\lambda}(\exp E)=\exp (\lambda E)$ for $\lambda >0$ and $E\in V_{1}$.

Lebesgue measure $\mathcal{L}^{n}$ is a Haar measure on $\bbG$. For $p\in\mathbb{G}$ define the \emph{left translation} $l_p\colon\mathbb{G}\to\bbG$ by $l_p(x)=px$. Then
\[\mathcal{L}^{n}(l_{p}(A))=\mathcal{L}^{n}(A) \qquad \mbox{and}\qquad \mathcal{L}^{n}(\delta_{\lambda}(A))=\lambda^{Q}\mathcal{L}^{n}(A)\]
for every $p\in \bbG$, $r>0$ and $A\subset \bbG$ measurable \cite[page 44]{BLU07}. Here $Q=\sum_{i=1}^{s}i\dim(V_{i})$ is the \emph{homogeneous dimension} of $\bbG$.

\begin{definition}\label{defdirectionalderivative}
Let $f\colon \bbG \to \bbR$ be a Lipschitz function, $x\in \bbG$ and $E\in V_{1}$. The \emph{directional derivative of $f$ at $x$ in direction $E$} is defined by
\[Ef(x)=\lim_{t\to 0} \frac{f(x\exp(tE))-f(x)}{t}\]
whenever the limit exists.
\end{definition}

The following definition is a special case of differentiability between Carnot groups, as proposed by Pansu in \cite{Pan89}.

\begin{definition}\label{pansudifferentiability}
A function $L\colon \bbG \to \bbR$ is \emph{$\bbG$-linear} if $L(xy)=L(x)+L(y)$ and $L(\delta_{r}(x))=rL(x)$ for all $x, y\in \bbG$ and $r>0$.

Let $f\colon \bbG\to \bbR$ and $x\in \bbG$. We say that $f$ is \emph{Pansu differentiable at $x$} if there is a $\bbG$-linear map $L \colon \bbG\to \bbR$ such that:
\[\lim_{y \to x} \frac{|f(y)-f(x)-L(x^{-1}y)|}{d(x,y)}=0.\]
In this case we say that $L$ is the \emph{Pansu differential} of $f$.
\end{definition}

\begin{remark}
It is readily seen that a Lipschitz function $f\colon \bbG\to \bbR$ is Pansu differentiable at $x\in \bbG$ if and only if for all $\xi\in\bbG$ the limit
\[
D f(x,\xi):=\lim_{t\to 0}\frac{f(x\delta_t\xi)-f(x)}{t}
\]
exists and defines a $\bbG$-linear map $\xi \mapsto D f(x,\xi)$, where the convergence in the limit is uniform with respect to $\xi$ whenever $\xi$ is restricted to a compact set. Moreover, $Df(x,\xi)=L(\xi)$ for all $\xi\in\bbG$.
\end{remark}

The following important result is proved in \cite{Pan89}.

\begin{theorem}[Pansu]\label{pansutheorem}
Every Lipschitz function $f\colon \bbG \to \bbR$ is Pansu differentiable Lebesgue almost everywhere.
\end{theorem}

Note that Theorem \ref{pansutheorem} also holds for Carnot group targets (and even for suitable infinite dimensional targets \cite{MR, MPS17}), but we will be concerned mainly with real-valued maps.

\begin{definition}
A set $N\subset \bbG$ is called a \emph{universal differentiability set (UDS)} if every Lipschitz map $f\colon \bbG \to \bbR$ is Pansu differentiable at a point of $N$.
\end{definition}

\begin{remark}\label{Hausone}
Let $N\subset\bbG$ be a UDS in a Carnot group. We claim the CC-Hausdorff dimension of $N$ is at least one. Assume it is strictly less than one and let $\pi_1:\bbG\to \bbR$ be the projection onto the first coordinate. Since $\pi_{1}$ is Lipschitz and the CC-Hausdorff dimension of $N$ is strictly less than one, $\pi_1(N)\subset \bbR$ has Lebesgue measure zero. Let $g:\bbR\to \bbR$ be a Lipschitz function (with respect to Euclidean distance) which is not differentiable at any point of $\pi_1(N)$. We claim the Lipschitz function $f:=g\circ \pi_1:\bbG\to \bbR$ is not differentiable at any $x\in N$, as the derivative $X_1f(y)$ does not exist for $y\in N$. Indeed, since $\pi_1(\exp tX_{1})=t$,
\begin{align*}
X_1f(y)&=\lim_{t\to 0}\frac{f(y\exp tX_{1})-f(y)}{t}\\
&=\lim_{t\to 0}\frac{g(\pi_1(y)+t)-g(\pi_1(y))}{t},
\end{align*}  
which does not exist by definition of $g$.
\end{remark}

Pansu's theorem implies that every positive measure subset of $\bbG$ is a UDS \cite{Mag01}. One of the themes of this paper is construction of much smaller UDS.

We now gather miscellaneous facts about lengths and distances in general Carnot groups which will be useful later in the paper.

In the first $r$ coordinates, the group operation is Euclidean and the dilation $\delta_{\lambda}$ acts by multiplication by $\lambda$. Equivalently $p(xy)=p(x)+p(y)$ and $p(\delta_{\lambda}(x))=\lambda p(x)$, where $p\colon \bbG\to \bbR^{r}$ is the projection defined by $p(x)=(x_{1},\ldots, x_{r})$.

The elements of $\mathfrak{g}$ are represented in coordinates as vector fields on $\bbR^{n}$. It can be shown that if $1\leq j\leq r$ then
\[X_{j}(x)=e_{j}+\sum_{i>r}^{n}q_{i,j}(x)e_{i},\]
where $q_{i,j}$ are polynomials satisfying various properties. In particular, $q_{i,j}(0)=0$. Using the above equation and the definition of exponential coordinates, it follows that $\exp (E)=E(0)$ for $E\in V_{1}$. Thus points $u=\exp (E)$ for some $E\in V_{1}$ are exactly those of the form $u=(u_{h},0)$ for some $u_{h}\in \bbR^{r}$.

It follows that if $E\in V_{1}$ then $p(E(x))$ is independent of $x\in \bbG$, so one can unambiguously define $p(E)\in \bbR^{r}$ for $E\in V_{1}$. In particular, $p(X_{j})=e_{j}$ for each $1\leq j\leq r$. Fixing an inner product norm $\omega$ on $V_{1}$ with respect to which $X_{1}, \ldots, X_{r}$ is an orthonormal basis, it follows that $\omega$ is equivalently given by $\omega(E)=|p(E)|$. 

From Definition \ref{horizontalcurve} we see that $L_{\bbG}(\gamma)$ is computed by integrating $|(p \circ \gamma)'(t)|$. In other words, $L_{\bbG}(\gamma)=L_{\bbE}(p \circ \gamma)$, where $L_{\bbE}$ is the Euclidean length of a curve in $\bbR^{r}$. This implies that $d(x,y)\geq |p(y)-p(x)|$, since the projection of a horizontal curve joining $x$ to $y$ is a curve in $\bbR^{r}$ joining $p(x)$ to $p(y)$. 

The following proposition will also be useful for estimating the CC distance \cite[Corollary 5.2.10 and Proposition 5.15.1]{BLU07}:

\begin{proposition}\label{euclideanheisenberg}
Let $\bbG$ be a Carnot group of step $s$ and $K\subset \bbG$ be a compact set. Then there exists a constant $C_{\mathrm{H}} \geq 1$ depending on $K$ such that:
\[ C_{\mathrm{H}}^{-1} |x-y|\leq d(x,y)\leq C_{\mathrm{H}}|x-y|^{\frac{1}{s}} \qquad \mbox{for all }x, y\in K.\]
\end{proposition}

The following Lemmas can be proved exactly as in \cite[Lemma 2.8]{PS16}, \cite[Lemma 2.9]{PS16} and \cite[Lemma 5.2]{PS16} respectively.

\begin{lemma}\label{horizontaldistances}
If $E\in V_{1}$ then:
\begin{itemize}
\item $|E(0)|=\omega(E)=d(E(0))$,
\item $d(x,x\exp tE)=t\omega(E)$ for any $x\in \bbG$ and $t\in \bbR$.
\end{itemize}
\end{lemma}




\begin{lemma}\label{lipschitzhorizontal}
Suppose $\gamma \colon I \to \bbG$ is a horizontal curve. Then:
\[\mathrm{Lip}_{\bbG}(\gamma) =  \mathrm{Lip}_{\bbE}(p \circ \gamma).\]
\end{lemma}



\begin{lemma}\label{lemmascalarlip}
Suppose $E\in V_{1}$ with $\omega(E)=1$ and let $L\colon \bbG \to \bbR$ be the function $L(x)=\langle x, E(0) \rangle$. Then:
\begin{enumerate}
\item $L$ is $\bbG$-linear and $\mathrm{Lip}_{\bbG}(L) = 1$,
\item for $x\in\bbG$ and $\tilde E\in V_{1}$:
\[\tilde{E}L(x)=L(\tilde{E}(0))=\langle p(\tilde{E}), p(E) \rangle.\]
\end{enumerate}
\end{lemma}

\subsection{Free Carnot groups of step 2}

We first give the more abstract but flexible definition of a free Carnot group via free-nilpotent Lie algebras. 
We then give the representation in coordinates of free Carnot groups of step 2 which we will use for most of the article. Recall that a homomorphism between Lie algebras is simply a linear map that preserves the Lie bracket. If it is also bijective than the map is an isomorphism. Free-nilpotent Lie algebras are then defined as follows (see Definition 14.1.1 in \cite{BLU07}).


\begin{definition}\label{freeliealgebra}
Let $r\geq 2$ and $s\geq 1$ be integers. We say that $\mathcal{F}_{r,s}$ is the \emph{free-nilpotent Lie algebra} with $r$ \emph{generators} $x_1, \ldots, x_{r}$ of \emph{step} $s$ if:
\begin{enumerate}
\item $\mathcal{F}_{r,s}$ is a Lie algebra generated by elements $x_1, \ldots, x_r$,
\item $\mathcal{F}_{r,s}$ is nilpotent of step $s$ (i.e., nested Lie brackets of length $s+1$ are $0$),
\item for every Lie algebra $\mathfrak{g}$ that is nilpotent of step $s$ and for every map $\Phi\colon \{x_1, \ldots, x_r\}\to \mathfrak{g}$, there is a homomorphism of Lie algebras $\tilde{\Phi}\colon \mathcal{F}_{r,s} \to \mathfrak{g}$ that extends $\Phi$, and moreover it is unique.
\end{enumerate}
\end{definition}

We next define free Carnot groups (see Definition 14.1.3 in \cite{BLU07}).

\begin{definition}\label{freecarnotgroup}
A \emph{free Carnot group} is a Carnot group whose Lie algebra is isomorphic to a free-nilpotent Lie algebra $\mathcal{F}_{r, s}$ for some $r\geq 2$ and $s\geq 1$. In this case the horizontal layer of the free Carnot group is isomorphic to the span of the generators of $\mathcal{F}_{r,s}$.
\end{definition}

Note that if two Carnot groups have isomorphic Lie algebras, then the Carnot groups themselves are isomorphic. Hence there is essentially one Carnot group corresponding to each free-nilpotent Lie algebra $\mathcal{F}_{r, s}$. By saying that two Carnot groups are isomorphic we simply mean that they are isomorphic as Lie groups, with an isomorphism that preserves the stratification. Such an isomorphism is called a \emph{Carnot group isomorphism}. Since Carnot groups are simply connected Lie groups, any homomorphism $\phi$ between their Lie algebras lifts to a Lie group homomorphism $F$ between the Carnot groups satisfying $dF=\phi$.

We now give an explicit coordinate representation of free Carnot groups of step 2. Fix an integer $r\geq 2$ and denote $n=r+r(r-1)/2$. In $\bbR^{n}$ denote the coordinates by $x_{i}$, $1\leq i\leq r$, and $x_{ij}$, $1\leq j<i\leq r$. Let $\partial_{i}$ and $\partial_{ij}$ denote the standard basis vectors in this coordinate system. Define $n$ vector fields on $\bbR^n$ by:
\[ X_{k}:=\partial_{k}+\sum_{j>k} \frac{x_{j}}{2}\partial_{jk}-\sum_{j<k}\frac{x_{j}}{2}\partial_{kj} \qquad \mbox{if }1\leq k \leq r,\]
\[X_{kj}:=\partial_{kj} \qquad \mbox{if } 1\leq j<k\leq r.\]

\begin{definition}\label{freegroupstep2}
Define the \emph{free Carnot group of step 2 and $r$ generators} by $\bbG_r:=(\bbR^{n}, \, \cdot)$, where the product $x\cdot y \in \bbG_r$ of $x,y\in \bbG_r$ is given by:
\[(x\cdot y)_{k} = x_{k}+y_{k} \qquad \mbox{if }1\leq k\leq r,\]
\[(x\cdot y)_{ij} = x_{ij}+y_{ij}+\frac{1}{2}(x_{i}y_{j}-y_{i}x_{j}) \qquad \mbox{if }1\leq j<i\leq r.\]
The Carnot structure of $\bbG_r$ is given by
\[V_{1}=\mathrm{Span} \{X_{k}\colon 1\leq k \leq r\} \mbox{ and }V_{2}=\mathrm{Span} \{X_{kj}\colon 1\leq j<k\leq r\}.\]
\end{definition}

Note that free Carnot groups of step 2 are exactly those that are isomorphic to a Carnot group $\bbG_r$ for some $r$. It is easily verified that for $1\leq j<k \leq r$ and $1\leq i\leq r$:
\[ [X_{k}, X_{j}]=X_{kj} \mbox{ and } [X_{i}, X_{kj}]=0.\]
The above presentation of $\bbG_{r}$ is in exponential coordinates of the first kind. The identity element of $\bbG_{r}$ is $0$ and the inverse of $x\in \bbG_{r}$ is $x^{-1}=-x$. 

Since $\bbG_{r}$ is simply $\bbR^{n}$ as a manifold, all the tangent spaces of $\bbG_{r}$ are naturally identified with $\bbR^{n}$. Recall that $v\in \bbR^{n}$ is horizontal at $p\in \bbG_{r}$ if $v=E(p)$ for some $E\in V_{1}$. The following proposition shows how horizontal curves in $\bbG_r$ are obtained by lifting curves in $\bbR^r$. This follows directly from the definitions of $X_{1}, \ldots, X_{r}$, so the proof is omitted.

\begin{proposition}\label{horizontalequation}
A vector $v\in \bbR^{n}$ is horizontal at $p\in \bbG_r$ if and only if for every $1\leq j<i\leq r$:
\[v_{ij}=\frac{1}{2}(p_{i}v_{j}-p_{j}v_{i}).\]
An absolutely continuous curve $\gamma \colon [a,b]\to \bbG_r$ is horizontal if and only if for every $1\leq j<i\leq r$:
\[\gamma_{ij}(t)=\gamma_{ij}(a) + \frac{1}{2}\int_{a}^{t} (\gamma_{i}\gamma_{j}'-\gamma_{j}\gamma_{i}')\]
for every $t\in [a,b]$.
\end{proposition}

If $(\gamma_{i}(a),\gamma_{j}(a))=(0,0)$ then $\frac{1}{2}\int_{a}^{t} (\gamma_{i}\gamma_{j}'-\gamma_{j}\gamma_{i}')$ can be interpreted as the signed area of the planar region enclosed by the curve $(\gamma_{i},\gamma_{j})|_{[a,t]}$ and the straight line segment joining $(0,0)$ to $(\gamma_{i}(t),\gamma_{j}(t))$.

\begin{definition}\label{horizontallift}
Suppose $\varphi\colon [a,b] \to \bbR^{r}$ is absolutely continuous, $p\in \bbG_r$ and $\varphi(a)=(p_{1},\ldots, p_{r})$. Define $\gamma\colon [a,b] \to \bbG_r$ by $\gamma_{i}=\varphi_{i}$ for $1\leq i \leq r$ and, for every $1\leq j<i\leq r$,
\[\gamma_{ij}(t)=p_{ij} + \frac{1}{2} \int_{a}^{t} (\varphi_{i}\varphi_{j}'-\varphi_{j}\varphi_{i}')\]
for every $t\in [a,b]$. We say that $\gamma$ is the \emph{horizontal lift} of $\varphi$ starting at $p$.
\end{definition}

Proposition \ref{horizontalequation} shows that the horizontal lift of an absolutely continuous curve in $\bbR^r$ is a horizontal curve in $\bbG_r$. 

Direct calculation shows that if $E\in V_{1}$ then
\begin{equation}\label{linestolines}x\exp (tE)= x (tE(0)) = x+tE(x)\end{equation}
for any $x\in \bbG_{r}$ and $t\in \bbR$. For the first term, notice if $E=\sum_{i} a_{i}X_{i} + \sum_{k>j}a_{kj}X_{kj}$ then $\exp(tE)$ is the element of $\bbR^{n}$ whose coordinates are the $ta_{i}$ and $ta_{kj}$. For the second and third term we are using the identification of tangent spaces with $\bbR^{n}$, so that $tE(0)$ and $tE(x)$ are elements of $\bbR^{n}$ calculated using the coordinate form of the $X_{i}$ and $X_{kj}$ given earlier. Equation \eqref{linestolines} expresses that `horizontal lines' are preserved by group translations. The crucial point is that in step two groups left translations are linear, with respect to exponential coordinates.

For computations it will be useful to use the Koranyi quasi-distance given by
\begin{align}\label{Koranyi} d_{K}(x)=(|x_{H}|^{4}+|x_{V}|^{2})^{1/4}, \quad \mbox{ where }x=(x_{H},x_{V})\in \bbR^r\times \bbR^{r(r-1)/2}, \end{align} 
and \[d_{K}(x,y)=d_{K}(x^{-1}y).\]
There exists a constant $C_{K}>1$ such that $C_{K}^{-1}d\leq d_{K}\leq C_{K}d$ (e.g. see \cite{BLU07}). Here $d$ is the CC distance in $\bbG_{r}$ induced by the basis $X_{1},\ldots, X_{r}$ of $V_{1}$. It follows that
\begin{equation}\label{sqrtdistance}
d((x_{H},x_{V}))\geq C_{K}^{-1}|x_{V}|^{1/2} \quad \mbox{for every} \quad (x_{H},x_{V})\in \bbG_{r}.
\end{equation}

\section{Maximality implies differentiability if the CC distance is differentiable}\label{strongmaximal}
In this section we show the equivalence of differentiability of the CC distance and `maximality implies differentiability' (Proposition \ref{equivalence}). We then show free Carnot groups of step 2 enjoy both properties (Theorem \ref{thmdifferentiabilityofdistance}).

Let $\bbG$ be a Carnot group represented in exponential coordinates and $r:=\dim V_{1}$.

\begin{lemma}\label{lipismaximal}
Let $f\colon \bbG \to \bbR$ be a Lipschitz map. Then:
\[\mathrm{Lip}_{\bbG}(f)=\sup\{|Ef(x)| \colon x\in \bbG, \, E\in V_{1}, \, \omega(E)=1, \, Ef(x) \mbox{ exists}\}.\]
\end{lemma}

\begin{proof}
Temporarily denote the right side of the above equality by $\mathrm{Lip}_{\mathrm{D}}(f)$. Note that $\mathrm{Lip}_{\mathrm{D}}(f)<\infty$, since for any $E\in V_{1}$ with $\omega(E)=1$ we have $|Ef(x)|\leq \mathrm{Lip}_{\bbG}(f)$. Let $x,y\in \bbG$. By Remark \ref{horizcurve}, $d(x,y)$ is equivalently given by the infimum of lengths of Lipschitz horizontal curves joining $x$ to $y$ which are each the concatenation of finitely many straight lines. Choose such a Lipschitz horizontal curve $\gamma \colon [0,L]\to \bbG$ such that $|(p\circ \gamma)'(t)|=1$ for almost every $t$. Let $G$ be the set of $t\in [0,L]$ for which:
\begin{itemize}
\item $(f\circ \gamma)'(t)$ exists,
\item $\gamma'(t)$ exists,
\item $\gamma'(t)\in \mathrm{Span}\{X_{i}(\gamma(t)) \colon 1\leq i \leq r\}$,
\item $|(p\circ \gamma)'(t)|=1$.
\end{itemize}
Since $\gamma$ is a horizontal curve and $f\circ \gamma$ is Lipschitz, we know that $G$ has full measure. We estimate as follows:
\begin{align*}
|f(x)-f(y)|&= \Big|\int_{0}^{L} (f\circ \gamma)' \Big|\\
&\leq L\ \sup \{ |(f\circ \gamma)'(t)|\colon t\in G\}\\
&\leq L\ \mathrm{Lip}_{\mathrm{D}}(f)
\end{align*}
using in the last line that $\gamma$ is a concatenation of lines and $|(p\circ \gamma)'(t)|=1$ so $(f\circ \gamma)'(t)$ is of the form $Ef(\gamma(t))$ for some $E\in V_{1}$ with $\omega(E)=1$. Taking an infimum over curves $\gamma$ then yields $\mathrm{Lip}_{\bbG}(f)\leq \mathrm{Lip}_{\mathrm{D}}(f)$.

For the opposite inequality fix $x\in \bbG$ and $E\in V_{1}$ such that $\omega(E)=1$ and $Ef(x)$ exists. Use Lemma \ref{horizontaldistances} to estimate as follows:
\begin{align*}
|Ef(x)|&= \Big| \lim_{t\to 0} \frac{f(x\exp tE)-f(x)}{t} \Big|\\
&\leq \limsup_{t \to 0} \frac{\mathrm{Lip}_{\bbG}(f)d(x,x\exp tE)}{t}\\
&=\mathrm{Lip}_{\bbG}(f).
\end{align*}
Hence $\mathrm{Lip}_{\mathrm{D}}(f) \leq \mathrm{Lip}_{\bbG}(f)$ which concludes the proof.
\end{proof}

We record the following property of the CC distance for later use.

\begin{lemma}\label{distanceinequality}
Let $u=\exp(E)$ for some $E\in V_{1}$. Then
\[d(uz) \geq d(u)+ \langle p(z), p(u)/d(u)\rangle \qquad \mbox{for any }z\in \bbG.\]
\end{lemma}

\begin{proof}
The point $u$ is of the form $u=(u_{h},0)$ for some $u_{h} \in \bbR^r$. We may assume that $d(u)=1$ since the general statement can be deduced using dilations. First recall $d(x)\geq |p(x)|$ for all $x\in \bbG$, while $d(u)=|p(u)|$ for our particular choice of $u$. Clearly also $\langle p(z),p(u)\rangle = \langle z,u\rangle$ for such $u$. We use Pythagoras' theorem and $d(u)=|p(u)|=1$ to estimate as follows:
\begin{align*}
d(uz) &\geq | p(u)(1+\langle p(z), p(u) \rangle ) + (p(z)-\langle p(z), p(u) \rangle p(u) )| \\
&\geq 1+\langle p(z), p(u)\rangle.
\end{align*}
\end{proof}

\begin{remark}\label{derivativeformula}
We can use Lemma \ref{distanceinequality} to show that if $u=\exp(E)$ for some $E\in V_{1}$ and the CC distance $d$ is differentiable at $u$, then the Pansu differential of $d$ at $u$ is given by $z\mapsto \langle p(z), p(u)/d(u)\rangle$.

To see this, suppose the CC distance $d$ is differentiable at $u$ as above with Pansu differential $L$. Since $L$ is $\bbG$-linear, \cite[Proposition 3.10]{Cas14} asserts that there exists $V\in \bbR^{r}$ such that $L(z)=\langle p(z),V\rangle$. In particular, for all $z\in \bbG$,
\[d(uz) \leq d(u)+ \langle p(z),V\rangle +o(d(z)).\]
Using Lemma \ref{distanceinequality}, we obtain that for $z\in \bbG$:
\[d(u)+ \langle p(z), p(u)/d(u)\rangle \leq d(u)+ \langle p(z),V\rangle +o(d(z)).\]
Hence
\[ \langle p(z)/d(z), p(u)/d(u)\rangle \leq \langle p(z)/d(z),V\rangle +o(d(z))/d(z).\]
Choosing $z=\exp(\pm tX_{i})$ for $i=1,\ldots, r$, with $t\downarrow 0$, gives $u_{i}/d(u)\leq V_{i}$ and $-u_{i}/d(u)\leq -V_{i}$ for $i=1,\ldots, r$. Hence $V=p(u)/d(u)$, as desired.
\end{remark}

\begin{theorem}\label{maximalimpliesdiff}
Let $\bbG$ be a Carnot group in which the CC distance $d$ is differentiable in horizontal directions.

Let $f\colon \bbG \to \bbR$ be Lipschitz, $x\in \bbG$ and $E\in V_{1}$ with $\omega(E)=1$. Suppose $Ef(x)$ exists and $Ef(x)=\mathrm{Lip}_{\bbG}(f)$. Then $f$ is differentiable at $x$ with derivative $w\mapsto \mathrm{Lip}_{\bbG}(f)L(w)$, where $L$ is the derivative of $d$ at $\exp(E)$.
\end{theorem}

\begin{proof}
Let $0<\varepsilon\leq 1/2$. Letting $L$ be the differential of the CC distance at $\exp(E)$, choose $0<\alpha \leq \varepsilon$ such that whenever $d(z) \leq \alpha$:
\[ d(\exp(E)z) - d(\exp(E)) \leq L(z) + \varepsilon d(z).\]
Use existence of $Ef(x)$ to fix $\delta>0$ such that whenever $|t|\leq \delta$:
\[|f(x\exp(tE))-f(x)-tEf(x)|\leq \alpha^{2}|t|.\]

Suppose that $0<d(w) \leq \alpha \delta$ and $t=\alpha^{-1}d(w)$. Then $0<t\leq \delta$, $d(\delta_{t^{-1}}(w))=\alpha$ and $2d(w)=2\alpha t\leq t$. Recall that $\omega(E)=1$ implies $d(\exp(E))=1$. We use also left invariance of the Carnot-Carath\'{e}odory distance to estimate as follows:
\begin{align*}
f(xw)-f(x) &= (f(xw) - f(x\exp(-tE))) + (f(x\exp(-tE)) - f(x))\\
&\leq \mathrm{Lip}_{\bbG}(f)d(xw,x\exp(-tE)) - tEf(x) + \alpha^{2}t\\
&= \mathrm{Lip}_{\bbG}(f)d(\exp(tE)w) - t\mathrm{Lip}_{\bbG}(f) + \alpha^{2}t\\
&= t \mathrm{Lip}_{\bbG}(f)\big( d(\exp(E)\delta_{t^{-1}}(w)) - d(\exp(E))\big) + \alpha^{2}t\\
&\leq t \mathrm{Lip}_{\bbG}(f)\big( L( \delta_{t^{-1}}(w)) + \varepsilon d(\delta_{t^{-1}}(w))\big) + \alpha^{2}t\\
&= \mathrm{Lip}_{\bbG}(f)L(w) + \varepsilon \mathrm{Lip}_{\bbG}(f) d(w) + \alpha d(w)\\
&\leq \mathrm{Lip}_{\bbG}(f)L(w) + \varepsilon (\mathrm{Lip}_{\bbG}(f) +1)d(w).
\end{align*}
For the opposite inequality we have:
\begin{align*}
f(xw)-f(x) &= (f(xw)-f(x\exp(tE))) + (f(x\exp(tE))-f(x))\\
&\geq -\mathrm{Lip}_{\bbG}(f)d(xw,x\exp(tE)) + tEf(x) - \alpha^{2}t\\
&= -\mathrm{Lip}_{\bbG}(f) d(\exp(tE)w^{-1})+ t\mathrm{Lip}_{\bbG}(f) - \alpha^{2}t\\
&= - t\mathrm{Lip}_{\bbG}(f)\big(d(\exp(E)\delta_{t^{-1}}(w^{-1})) - d(\exp(E))\big) - \alpha^{2}t\\
&\geq -t\mathrm{Lip}_{\bbG}(f)\big( L(\delta_{t^{-1}}(w^{-1})) + \varepsilon d(\delta_{t^{-1}}(w^{-1})) \big) - \alpha^{2}t\\
&= -\mathrm{Lip}_{\bbG}(f)L(w^{-1}) - \varepsilon \mathrm{Lip}_{\bbG}(f)d(w^{-1}) - \alpha^{2}t\\
&= \mathrm{Lip}_{\bbG}(f)L(w) - \varepsilon\mathrm{Lip}_{\bbG}(f)d(w)-\alpha d(w)\\
&\geq \mathrm{Lip}_{\bbG}(f)L(w) - \varepsilon(\mathrm{Lip}_{\bbG}(f)+1)d(w).
\end{align*}
This shows that $d(w) \leq \alpha\delta$ implies:
\[|f(xw)-f(x)-\mathrm{Lip}_{\bbG}(f)L(w) |\leq \varepsilon (\mathrm{Lip}_{\bbG}(f)+1) d(w).\]
Hence $f$ is Pansu differentiable at $x$ with differential $w\mapsto \mathrm{Lip}_{\bbG}(f)L(w)$.
\end{proof}

The CC distance satisfies $\mathrm{Lip}_{\bbG}(d)=1$ in any Carnot group. Hence the following proposition shows that the CC distance always admits a maximal directional derivative.

\begin{proposition}\label{Ed=1}
Suppose $\bbG$ is a general Carnot group and $u=\exp(E)$ for some $E\in V_{1}$ with $\omega(E)=1$. Then the directional derivative of the distance $d$ satisfies $Ed(u)=1$.
\end{proposition}

\begin{proof}
As an easy consequence of the Baker-Campbell-Hausdorff formula:
\begin{align*}
\frac{d(u\exp tE)-d(u)}{t}&= \frac{d(\exp E \exp tE)-d(\exp E)}{t}\\
&=\frac{d(\exp(t+1)E)-d(\exp E)}{t}\\
&=\frac{(t+1)d(\exp E)-d(\exp E)}{t}\\
&=d(\exp E)\\
&=1.
\end{align*}
This implies $Ed(u)=1$ as desired.
\end{proof}

\begin{proof}[Proof of Proposition \ref{equivalence}:]
As a consequence of Proposition \ref{Ed=1}, if in a Carnot group $\bbG$ the existence of a maximal directional derivative implies Pansu differentiability (i.e. the analogue of Theorem \ref{maximalimpliesdiff} holds), then the distance is Pansu differentiable in horizontal directions. Combining this with Theorem \ref{maximalimpliesdiff} the conclusion follows.
\end{proof}

Now recall that $\bbG_{r}$ is the free Carnot group of step $2$ and rank $r$, represented by $\bbR^{n}$ with $n=r+r(r-1)/2$. We will now show that the CC distance in $\bbG_{r}$ is differentiable in horizontal directions, hence maximality implies differentiability in $\bbG_{r}$. Let $P=r(r-1)/2$ be the number of vertical coordinates of $\bbG_{r}$.

\begin{lemma}\label{curve}
Fix $y\in \bbG_{r}$ such that $y_{1}>0$ and $y_{i}=0$ for $i>1$. Define the numbers $A=\max_{2\leq i\leq r}|y_{i1}|$ and $B=\max_{i>j>1}|y_{ij}|$. Then there exists a Lipschitz horizontal curve $\gamma\colon [0,1]\to \bbG_{r}$ which is a concatenation of horizontal lines such that:
\begin{enumerate}
\item $\gamma(0)=0$ and $\gamma(1)=y$,
\item The Lipschitz constant of $\gamma$ satisfies
\[\mathrm{Lip}_{\bbG}(\gamma)\leq y_{1} \max \left\{ \left(1+\frac{16P^4A^2}{y_{1}^4}\right)^{1/2},\, \left(1+\frac{288P^2B}{y_{1}^2}\right)^{1/2} \right\},\]
\item $\gamma'(t)$ exists for all $t\in [0,1]$ except finitely many points and satisfies
\[ |(p\circ \gamma)'(t) - p(y)|\leq \max \left\{ \frac{4P^2 A}{y_{1}},\, 24P\sqrt{B} \right\}.\]
\end{enumerate}
\end{lemma}

\begin{proof}
Divide $[0,1]$ into $P$ subintervals $I_{q}=[q/P,(q+1)/P]$ for $0\leq q\leq P-1$. To each interval $I_{q}$ we assign a vertical coordinate $x_{ij}$, $1\leq j<i\leq r$, in such a way that every vertical coordinate is assigned to some interval $I_{q}$. 

We define the curve $\gamma$ inductively. Let $\gamma(0)=0$ and suppose $\gamma$ has already been defined on intervals $I_{1}\cup \ldots \cup I_{q-1}$ for some $0\leq q\leq P-1$, where $I_{-1}=\{0\}$ and $\gamma(0)=0$. We also assume $\gamma$ has been constructed so that $\gamma(q/P)$ satisfies:
\begin{enumerate}
\item $\gamma_{1}(q/P)=y_{1}q/P$,
\item $\gamma_{i}(q/P)=0$ for $i>1$,
\item $\gamma_{ij}(q/P)=y_{ij}$ if the coordinate $x_{ij}$ was assigned to one of the intervals $I_{1}, \ldots, I_{q-1}$,
\item $\gamma_{ij}(q/P)=0$ if the coordinate $x_{ij}$ was not assigned to one of the intervals $I_{1}, \ldots, I_{q-1}$.
\end{enumerate}
Let $x_{ij}$ be the vertical coordinate assigned to $I_{q}$. The construction of $\gamma$ on $I_{q}$ depends on whether $j=1$ or $j\neq 1$.

\begin{case}[$j=1$]
Let $\lambda=4P^{2}y_{i1}/y_{1}$. Since $\gamma(q/P)$ is already defined, we may define $\gamma$ on $I_{q}$ by:
\[\gamma'(t)=
\begin{dcases}
y_{1}X_{1}(\gamma(t))+\lambda X_{i}(\gamma(t)), & \frac{q}{P}< t < \frac{q}{P} + \frac{1}{2P} \\
y_{1}X_{1}(\gamma(t))-\lambda X_{i}(\gamma(t)), & \frac{q}{P} + \frac{1}{2P} < t < \frac{q+1}{P}
\end{dcases}
\]
Clearly $\gamma'(t)$ exists for all but finitely many points. For such points we have that $(p\circ \gamma)'=y_{1}e_{1}\pm \lambda e_{i}$,
\begin{equation}\label{lipcase1}
|(p\circ \gamma)'(t)|=(y_{1}^2+\lambda^2)^{1/2}=y_{1}\left(1+\frac{16P^{4}y_{i1}^{2}}{y_{1}^4}\right)^{1/2}
\end{equation}
and
\begin{equation}\label{derivcase1}
|(p\circ \gamma)'(t)-p(y)| = |\lambda|=\frac{4P^{2}|y_{i1}|}{y_{1}}.
\end{equation}
It follows from the definition of the vector fields that
\[\gamma((q+1)/P)=\gamma(q/P)VW,\]
where the non-zero coordinates of $V, W\in \bbG_{r}$ are given by:
\begin{itemize}
\item $V_{1}=y_{1}/2P,\quad V_{i}=\lambda/2P$,
\item $W_{1}=y_{1}/2P,\quad W_{i}=-\lambda/2P$,
\end{itemize}
Notice $VW\in \bbG_{r}$ satisfies $(VW)_{1}=y_{1}/P$, $(VW)_{i1}=\lambda y_{1}/4P^{2}=y_{i1}$, and all other coordinates are zero. It follows that $\gamma((q+1)/P)$ has properties (1)--(4) with $q$ replaced by $q+1$.
\end{case}

\begin{case}[$j\neq 1$]
Let $\lambda=6P\sqrt{|y_{ij}|}$ and $\mu=-6P\ \mathrm{sign}(y_{ij})\sqrt{|y_{ij}|}$. We divide $I_{q}$ into six equal subintervals of length $1/6P$ and denote their interiors by $I_{q}^{l}$. Since $\gamma(q/P)$ is already specified, we may define $\gamma$ on $I_{q}$ by:
\[\gamma'(t)=
\begin{dcases}
y_{1}X_{1}(\gamma(t))+\lambda X_{j}(\gamma(t)), & t\in I_{q}^{1} \\
y_{1}X_{1}(\gamma(t))+\mu X_{i}(\gamma(t)), & t\in I_{q}^{2}\\
y_{1}X_{1}(\gamma(t))-\lambda X_{j}(\gamma(t)),  & t\in I_{q}^{3}\\
y_{1}X_{1}(\gamma(t))-\mu X_{i}(\gamma(t)), & t\in I_{q}^{4}\\
y_{1}X_{1}(\gamma(t))-2\lambda X_{j}(\gamma(t))-2\mu X_{i}(\gamma(t)), & t\in I_{q}^{5}\\
y_{1}X_{1}(\gamma(t))+2\lambda X_{j}(\gamma(t))+2\mu X_{i}(\gamma(t)), & t\in I_{q}^{6}
\end{dcases}\]

Clearly $\gamma'(t)$ exists for all but finitely many points. For such points $(p\circ \gamma)'$ is equal to one of the following:
\[ y_{1}e_{1}\pm \lambda e_{j},\quad y_{1}e_{1}\pm \mu e_{i}, \quad y_{1}e_{1}\pm (2\lambda e_{j}+2\mu e_{i}).\]
Hence, if $\gamma'(t)$ exists,
\begin{align}\label{lipcase2}
|(p\circ \gamma)'(t)|&\leq (y_{1}^2 + 4\lambda^2 + 4\mu^2)^{1/2}\nonumber \\
&=(y_{1}^2 + 288P^2 |y_{ij}|)^{1/2}\nonumber \nonumber \\
& = y_{1}\left(1+\frac{288P^2 |y_{ij}|}{y_{1}^2}\right)^{1/2}
\end{align}
and
\begin{equation}\label{derivcase2}
|(p\circ \gamma)'(t)-p(y)|\leq 2|\lambda|+2|\mu|\leq 24P\sqrt{|y_{ij}|}.
\end{equation}

To compute $\gamma((q+1)/P)$ we notice
\[\gamma((q+1)/P)=\gamma(q/P)V^{1}V^{2}V^{3}V^{4}V^{5}V^{6},\]
where the non-zero coordinates of $V^{l}\in \bbG_{r}$ are given by
\begin{itemize}
\item $V^{1}_{1}=y_{1}/6P,\quad V^{1}_{j}=\lambda/6P$,
\item $V^{2}_{1}=y_{1}/6P,\quad V^{2}_{i}=\mu/6P$,
\item $V^{3}_{1}=y_{1}/6P,\quad V^{3}_{j}=-\lambda/6P$,
\item $V^{4}_{1}=y_{1}/6P,\quad V^{4}_{i}=-\mu/6P$,
\item $V^{5}_{1}=y_{1}/6P,\quad V^{5}_{j}=-2\lambda/6P,\quad V^{5}_{i}=-2\mu/6P$,
\item $V^{6}_{1}=y_{1}/6P,\quad V^{6}_{j}=2\lambda/6P,\quad V^{6}_{i}=2\mu/6P$.
\end{itemize}
We avoid giving all intermediate steps, but computation using the group law shows that $W^{1}=V^{1}V^{2}V^{3}V^{4}$ and $W^{2}=V^{5}V^{6}$ have non-zero coordinates:
\begin{itemize}
\item $W^{1}_{1}=2y_{1}/3P, \quad W^{1}_{j1}=\lambda y_{1}/18P^{2}, \quad W^{1}_{i1}=\mu y_{1}/18P^{2}, \quad W^{1}_{ij}=-\lambda\mu/36P^{2}$,
\item $W^{2}_{1}=y_{1}/3P, \quad W^{2}_{j1}=-\lambda y_{1}/18P^{2}, \quad W^{2}_{i1}=-\mu y_{1}/18P^{2}, \quad W^{2}_{ij}=0$.
\end{itemize}
It follows that $(W^{1}W^{2})_{1}=y_{1}/P$, $(W^{1}W^{2})_{ij}=-\lambda\mu/36P^{2}=y_{ij}$, and all other coordinates are zero. Hence $\gamma((q+1)/P)$ again has properties (1)--(4) with $q$ replaced by $q+1$.
\end{case}

The cases above define a Lipschitz horizontal curve $\gamma$ on $[0,1]$ which is differentiable at all but finitely many points. Clearly $\gamma(0)=0$. Applying (1)--(4) with $q=P$ gives $\gamma(1)=y$, since every coordinate $x_{ij}$ was assigned to some interval $I_{q}$. This gives conclusion (1). Using \eqref{lipcase1} and \eqref{lipcase2} together with Lemma \ref{lipschitzhorizontal} gives conclusion (2). Putting together \eqref{derivcase1} and \eqref{derivcase2} gives conclusion (3).
\end{proof}

The next lemma enables us to choose convenient coordinates in $\bbG_{r}$. The proof is a slight adaption of \cite[page 7]{LS16}. By a \emph{group isometric isomorphism} we simply mean a Carnot group isomorphism which preserves the distance.

\begin{lemma}\label{existF}
Suppose $y\in \bbG_{r}$ with $L=|p(y)|\neq 0$. Then there exists a group isometric isomorphism $F\colon \bbG_{r}\to \bbG_{r}$ such that $F_{1}(y)=L$ and $F_{i}(y)=0$ for $i>1$. Such a map can be chosen of the form $F(x,y)=(A(x),B(y))$, where $A\colon \bbR^{r}\to \bbR^{r}$ is a linear isometry and $B\colon \bbR^{P}\to \bbR^{P}$ is linear.
\end{lemma}

\begin{proof}
Let $E=y_{1}X_{1}+\cdots + y_{r}X_{r}\in V_{1}$. Choose a linear bijection $\Phi \colon V_{1} \to V_{1}$ such that $\Phi(E)=LX_{1}$ which preserves the inner product induced by the basis $X_{1}, \ldots, X_{r}$ of $V_{1}$. Since $\mathbb{G}_{r}$ is free-nilpotent, $\Phi$ extends to an isomorphism of the Lie algebra of $\mathbb{G}_{r}$ (Lemma 14.1.4 \cite{BLU07}). Since Carnot groups are simply connected Lie groups, such an isomorphism lifts to a Lie group isomorphism $F\colon \mathbb{G}_{r}\to \mathbb{G}_{r}$ satisfying $dF|_{V_{1}}=\Phi$. The map $F$ is smooth and $dF$ preserves $V_{1}$ and $V_{2}$, in particular $F$ sends horizontal/smooth curves to horizontal/smooth curves. The definition of $F$ and of $X_{1}, \ldots, X_{r}$ imply that $F$ has the form $F(x,y)=(A(x),B(y))$, where $A\colon \mathbb{R}^{r} \to \mathbb{R}^{r}$ is a linear isometry and $B\colon \mathbb{R}^{n-r}\to \mathbb{R}^{n-r}$ is linear. That $F$ is an isometry follows from the fact that $\Phi$ preserves the inner product of $V_1$.
\end{proof}

\begin{theorem}\label{thmdifferentiabilityofdistance}
The CC distance in $\bbG_{r}$ is differentiable in horizontal directions.
\end{theorem}

\begin{proof}
Let $u=\exp E$ for some $E\in V_{1}$. Equivalently, $u=(u_{h},0)$ for some choice $u_{h}\in \bbR^{r}\setminus \{0\}$. We may assume $d(u)=1$. By Lemma \ref{distanceinequality}, it suffices to show that $d(uz)\leq d(u)+\left\langle p(z), p(u) \right\rangle+ o(d(z))$ as $z\to 0$. Assume $d(z)\leq 1/2$ and let $L=|p(uz)|\neq 0$. Using $|p(u)|=1$ and $|p(z)|\leq d(z)\leq 1/2$ we see $1/2\leq L\leq 2$. By Lemma \ref{existF}, there exists a group isometric isomorphism $F\colon \bbG_{r}\to \bbG_{r}$ such that $F_{1}(uz)=L$ and $F_{i}(uz)=0$ for $i>1$. Write $F=(p\circ F,v)$ for some function $v\colon \bbG_{r}\to \bbR^{P}$. Then
\[(L,0,\ldots,0,v(uz))=F(uz)=F(u)F(z).\] 
The group law is additive in the first $r$ coordinates, so $F_{1}(u)+F_{1}(z)=L$ while $F_{i}(u)+F_{i}(z)=0$ for $i>1$. Notice the form of the map $F$ in Lemma \ref{existF} implies $F_{ij}(u)=0$ for all $i, j$. Using the definition of the group law gives, for all $i>1$:
\begin{align}\label{firstvert}
F_{i1}(uz)&=F_{i1}(u)+F_{i1}(z)+\frac{1}{2}(F_{i}(u)F_{1}(z)-F_{1}(u)F_{i}(z))\\
&=F_{i1}(z)+\frac{1}{2}(-F_{i}(z)F_{1}(z)-(L-F_{1}(z))F_{i}(z))\\ \nonumber
&=F_{i1}(z)-\frac{1}{2}LF_{i}(z).\nonumber
\end{align}
Similarly, for all $i>j>1$:
\begin{align}\label{secovert}
F_{ij}(uz)= F_{ij}(z).
\end{align}
Using \eqref{firstvert}, we estimate as follows
\begin{align}\label{crt}
A:=\max_{2\leq i\leq r} |F_{i1}(uz)|&\leq |v(z)|+\frac{L}{2} |p(F(z))|\\
\nonumber
&\leq C d(F(z))^2+\frac{L}{2}d(F(z))\\
\nonumber
&= C d(z)^2+\frac{L}{2}d(z)\\
\nonumber
&\leq C d(z),
\end{align}
where in the equality above we used that $F$ is an isometry. By \eqref{secovert}, we also have
\begin{align}\label{crt2}
B:=\max_{i>j>1} |F_{ij}(uz)|\leq |v(F(z))|\leq C d(z)^2.
\end{align}
Lemma \ref{curve} yields
\begin{align}\label{stm}
d(uz)=d(F(uz))&\leq L\max \left\{ \left(1+\frac{16P^4A^2}{L^4}\right)^{1/2},\, \left(1+\frac{288P^2B}{L^2}\right)^{1/2} \right\}.
\end{align}
Using \eqref{crt}, \eqref{crt2} and $1/2<L<2$ we get
\begin{align*}
&L\left(1+\frac{16P^4A^2}{L^4}\right)^{1/2}\leq L\left(1+256C^2 P^4 d(z)^2\right)^{1/2}\leq L+o(d(z)), \\
&L\left(1+\frac{288P^2B}{L^2}\right)^{1/2}\leq L\left(1+1200CP^4 d(z)^2\right)^{1/2}\leq L+o(d(z)).
\end{align*} 
Therefore \eqref{stm} gives $d(uz)\leq  L+o(d(z))$. 
To conclude the proof of (2) we show that $L\leq 1+\langle p(z),p(u)\rangle + o(d(z))$. We estimate as follows:
\begin{align*}
L&=|p(u)+p(z)|\\
&=|p(u)(1+\langle p(z),p(u)\rangle) + (p(z)-\langle p(z), p(u)\rangle p(u)) |\\
&=\left( (1+\langle p(z),p(u) \rangle)^{2} + |p(z)-\langle p(z),p(u) \rangle p(u)|^{2} \right)^{\frac{1}{2}}\\
&\leq \left( (1+\langle p(z),p(u) \rangle)^{2} + 4d(z)^{2} \right)^{\frac{1}{2}}\\
&=(1+\langle p(z),p(u) \rangle)\left(1+\frac{4d(z)^{2}}{(1+\langle p(z),p(u) \rangle)^{2}}\right)^{\frac{1}{2}}\\
\end{align*}
The claim then follows since $d(z)/(1+\langle p(z), p(u) \rangle)\leq 2d(z)\to 0$ as $d(z)\to 0$.
\end{proof}

\section{Maximality does not imply differentiability in the Engel group}\label{engelsection}
In this section we show that existence of a maximal directional derivative does not imply differentiability in the Engel group, a Carnot group of step 3 (Theorem \ref{Conterex:Engel}). A representation of the Engel group in $\mathbb R^{4}$ with respect to exponential coordinates of the second kind is the following. Let  $x_1, x_2, x_3, x_4$ be the coordinates of  a point $x\in \mathbb R^4$. Then the Engel group law is given by
$$x\cdot y=
\Big(x_1+y_1,\, x_2+y_2,\,
x_3+y_3 -x_1y_2,\,
x_4+y_4 - x_1 y_3 +\frac{1}{2} x_1^2y_2
\Big).$$

The horizontal bundle is spanned by the two left-invariant vector fields:     
  $$X_{1} = \partial_{1}\qquad \text{ and}\qquad  
X_{2} =\partial_{2} -x_{1}\partial_{3} + \frac{x_{1}^{2}}{2}\partial_{4} .
$$
Other elements completing these to a  basis of the Lie algebra are the following: 
$$
X_{3}  =  \partial_3-  x_{1}\partial_{4} \qquad \text{ and}\qquad 
X_{4}=\partial_4.$$
With respect to this basis, the non-trivial bracket relations are:
$$X_3=[X_2,X_1] \qquad \text{ and}\qquad 
X_4=[X_3,X_1].$$

The vector $X_2$ is called the {\em abnormal direction}, since the curves of the form $t\mapsto p\exp(tX_2)$, $p\in \mathbb R^4$, are the abnormal curves (i.e. are singular points of the endpoint map). Because of the presence of these abnormal curves, the CC distance (defined using the basis $X_{1}, X_{2}$ of the horizontal layer) behaves irregularly. Indeed, the distance function from the origin is not differentiable at any of the points $\exp(\mathbb R X_2)$. This result is probably not new to experts, but for completeness we will provide a complete proof later in this section.

\begin{lemma} \label{lemma:Engel}
In the Engel group, with coordinates as above, there exists $C>0$ such that 
\begin{equation}\label{lemma:Engel:eq}
d( ( 0,0,0,0) , (0,1,0,\varepsilon) ) \geq 1+ C |\varepsilon|^{1/3} \qquad \text{ for } \varepsilon \in (-1,0).
\end{equation}
\end{lemma}

The above bound will be enough to show that a maximal directional derivative does not imply differentiability in the Engel group. Nonetheless we observe that \eqref{lemma:Engel:eq} is actually sharp, i.e. we have
$$|d( ( 0,0,0,0) , (0,1,0,\varepsilon) ) -1 |\simeq |\varepsilon|^{1/3} \qquad \text{ as } \varepsilon \to 0^-.$$
Indeed, by triangle inequality and left invariance, for all $\varepsilon\in \bbR$ we have:
\begin{align*}
d( ( 0,0,0,0) , (0,1,0,\varepsilon) ) &\leq d( ( 0,0,0,0) , (0,1,0,0) ) +d( ( 0,1,0,0) , (0,1,0,\varepsilon) )\\
&=1+ d( ( 0,0,0,0) , (0,0,0,\varepsilon) ) \\
&= 1+ C' |\varepsilon|^{1/3}.
\end{align*}

Recall that the CC distance in a Carnot group has Lipschitz constant $1$.

\begin{theorem}\label{Conterex:Engel}
Let $\bbG$ be the Engel group. Then the CC distance $d\colon \bbG\to \bbR$ from the origin is not differentiable at $\bar p = (0,1,0,0)=\exp(X_2)$, but still $X_2 d (\bar p)=1$.
\end{theorem}

\begin{proof}
Suppose $d$ is differentiable at $\bar{p}$ with differential $L\colon \bbG \to \bbR$. Then $L$ must have the form $L(h)=\left( X_{1}d(\bar{p}),X_{2}d(\bar{p})\right)\cdot \left( h_{1}, h_{2}\right)$. In particular, $L((0,0,0,\varepsilon))=0$ for any $\varepsilon \in \bbR$. Using Lemma \ref{lemma:Engel}, for $\varepsilon \in (-1,0)$:
\begin{align*}
d(\bar{p}\cdot (0,0,0,\varepsilon))-d(\bar{p})-L((0,0,0,\varepsilon)) &= d((0,1,0,\varepsilon))-1\\
&\geq C|\varepsilon|^{1/3}\\
&\geq Cd((0,0,0,\varepsilon)),
\end{align*}
where the constant $C$ may change from line to line. Hence $d$ is not differentiable at $\bar{p}$. Nonetheless, we have:
\[X_2 d (\bar p)  = \lim_{t \to 0} \frac{d((0,1+t,0,0))-d((0,1,0,0))}{t} = 1.\]

\end{proof}

\subsection{The Martinet distribution and a proof of Lemma ~\ref{lemma:Engel}}
Lemma ~\ref{lemma:Engel} is probably known to many experts. It can be obtained from the recent work of Sachkov \cite{AS15}.
Nonetheless, to obtain a proof with minimal effort we use the sub-Riemannian geometry of the Martinet distribution, which is more studied \cite{ABCK97}.

By the Martinet geometry we mean the sub-Riemannian geometry on $\mathbb R^3$, with coordinates $x,y,z$, for which a horizontal orthonormal frame is given by 
$$X = \partial_{x}+ \frac{y^{2}}{2}\partial_{z}\qquad \text{and}\qquad Y  =\partial_{y} .$$
This geometry is called the `flat case' in \cite{ABCK97}. The CC distance is defined using horizontal curves, just like in the case of Carnot groups. From \cite{ABCK97}, we have a clear picture of the sub-Riemannian sphere. 
 In particular, we know that the sphere makes a partial cusp, in the sense that one of the geodesics (the abnormal curve) is tangent to the sphere. Moreover, the tangency is of order 3. 
 
 In fact, from Proposition 4.15 in \cite{ABCK97} we have a quite precise bound. Namely, we know that the intersection of the  half plane  $\{ y=0, z\geq0 \}$ with the sphere $S(0, r)$, $r>0$,   is a curve    whose  graph    is given by
\begin{equation}\label{formula:Agrachev}
z=z(x)= \dfrac{r^3}{6} \left(\dfrac{x+r}{2r} \right)^3 +G \left(\dfrac{x+r}{2r} \right),
\qquad \text{as } x\to -r^+,
\end{equation}
where $G$ is  function of the form $G(t) = -4r^3 t^3 e^{-2/t}  +  o( t^3 e^{-2/t})$. This is an expansion at the point $(-r,0,0) \in S(0, r)$ in the half plane, see Figure 3 in \cite{ABCK97}.

The link between the  Engel geometry  and the Martinet geometry is that the latter is a metric submersion of the first one. More precisely, there exists a map $F:\mathbb R^4\to \mathbb R^3$ for which 
$F_* X_1= Y $ and $F_* X_2= X$. Therefore, in particular, the map $F$ is 
1-Lipschitz. 
In the coordinates that we chose the map $F$ is given by $F(x_1 , x_2 ,x_3, x_4)= ( x_2 ,x_1 ,x_4)$, i.e., we forget the third coordinate and we swap the first two coordinates (the reason for this second choice is to fit with the coordinate system from \cite{ABCK97}).
For example, here is the verification of $F_* X_2= X$, i.e. that, $dF (X_2) = X\circ F $:
\[dF_x X_2(x)=\left[ \begin{matrix}
0&1&0&0\\
1&0&0&0\\
0&0&0&1
\end{matrix}\right]\left[ \begin{matrix}
0\\
1\\
-x_1\\
x_1^2/2
\end{matrix}\right]=
\left[ \begin{matrix}
1\\
0\\ 
(F(x)_2)^2/2
\end{matrix}\right]= X ( F(x)). 
\]

\begin{proof}[Proof of Lemma \ref{lemma:Engel}]
We now use \eqref{formula:Agrachev}.
Let $p:=(-1,0,\zeta)$ with $\zeta$ positive and small. Let $r:= d(0,p)$. As $\zeta \to 0$ we have $r\to 1$, so we may assume that $r\in(1/2, 2)$.
In the expansion \eqref{formula:Agrachev}, we are considering $ x\to -r^+$, hence we assume $x+r\geq 0$, thus the argument of the
 the function   $G$ in \eqref{formula:Agrachev} is non-negative.
 Moreover, for small $t\geq 0$ we can bound
  $|G(t)|\leq C t^3$ for  some $C$. In the rest of the argument we shall call $C$ the generic constant that may change from line to line.
  Therefore, for  $ x\to -r^+$, for some $C$, we bound
$$0\leq z(x)\leq 
C\left( {x+r}  \right)^3 .$$
Hence, taking the values $x=-1$ and $z=\zeta>0$, we get
$$\zeta
 \leq C  (r-1)^3.
 $$
We then conclude that 
$$d(0,p) = r \geq 1+ C \sqrt[3]{\zeta}.$$
Notice that $G(0,-1,0, \zeta)= (-1,0,\zeta)=p$. Hence, since $G$ is 1-Lipschitz
$$1+ C \sqrt[3]{\zeta}
\leq d (  0  ,  (-1,0,\zeta))
\leq d( 0 , (0,-1,0, \zeta)).$$
By left-invariance of the distance in the Engel group we conclude that, for $\zeta$ positive and small, we have
 $$d( 0 , (0,1,0, -\zeta)) = d( 0 , (0,-1,0, \zeta)) \geq  1+ C \sqrt[3]{\zeta}.  $$
Setting $\varepsilon=-\zeta$ we have the conclusion.
\end{proof}

\section{A universal differentiability set in free Carnot groups of step 2}\label{sectionUDS}

In this section we extend the construction of a measure zero UDS in \cite{PS16} to free Carnot groups of step 2. We avoid repeating much of the work of \cite{PS16}, but show how to generalize the necessary geometric lemmas which rely on the structure of the individual Carnot group (Lemma \ref{closedirectioncloseposition} and Lemma \ref{curveforalmostmax}). After defining a Hausdorff dimension one set $N$ containing suitably many curves in Lemma \ref{uds}, the argument that $N$ is a universal differentiability set is almost exactly as in \cite{PS16} (`almost maximality' of directional derivatives implies differentiability in Proposition \ref{almostmaximalityimpliesdifferentiability} and construction of an `almost maximal' directional derivative in Proposition \ref{DoreMaleva}).

The first geometric lemma is a generalization of \cite[Lemma 5.1]{PS16} in the Heisenberg group to step 2 free Carnot groups.

\begin{lemma}\label{closedirectioncloseposition}
Given $S>0$, there is a constant $C_{\mathrm{a}}=C_{\mathrm{angle}}(S)\geq 1$ for which the following is true. Suppose:
\begin{itemize}
\item $g, h \colon I \to \bbG_{r}$ are Lipschitz horizontal curves with $\mathrm{Lip}_{\bbG_{r}}(g), \mathrm{Lip}_{\bbG_{r}}(h)\leq S$,
\item $g(c)=h(c)$ for some $c\in I$,
\item there exists $0\leq A\leq 1$ such that $|(p\circ g)'(t) - (p\circ h)'(t)| \leq A$ for almost every $t\in I$.
\end{itemize}
Then $d(g(t), h(t))\leq C_{\mathrm{a}}\sqrt{A}|t-c|$ for every $t\in I$.
\end{lemma}

\begin{proof}
Assume $c=0\in I$ and, using left group translations, $g(0)=h(0)=0$. We estimate using the equivalent Koranyi quasi-distance \eqref{Koranyi}:
\begin{align}\label{angleestimate}
d(g(t),h(t))&\leq C_K
d_{K}(g(t),h(t)) \\
\nonumber
&\leq C_K \sum_{i=1}^r \left|h_i(t)-g_i(t)\right|\\
\nonumber
&\qquad +C_K \sum_{1\leq j<i\leq r}\left|h_{ij}(t) - g_{ij}(t) + \frac{1}{2}(g_{i}(t)h_{j}(t)-g_{j}(t)h_{i}(t))\right|^{\frac{1}{2}}
\end{align}
where $C_K\geq 1$ is a constant depending only on the Carnot group.

Let $1\leq j\leq r$. Using $|(p\circ g)'(t) - (p\circ h)'(t)| \leq A$ almost everywhere implies $|h_{j}(t)-g_{j}(t)|\leq A|t|$ for every $t\in I$. Lemma \ref{lipschitzhorizontal} and $\mathrm{Lip}_{\bbG}(g)\leq S$ give the inequality $|g_{j}'|\leq \mathrm{Lip}_{\mathbb{E}}(g_{j})\leq S$. Using also $g(0)=0$ then gives $|g_{j}(t)|\leq S|t|$ for $t\in I$. For $1\leq i\leq r$ and $t\in I$:
\begin{align*}
|g_{i}(t)h_{j}(t)-g_{j}(t)h_{i}(t)| &= |g_{i}(t)(h_{j}(t)-g_{j}(t)) + g_{j}(t)(g_{i}(t)-h_{i}(t))| \\
&\leq S|t| |h_{j}(t)-g_{j}(t)| + S|t| |g_{i}(t)-h_{i}(t)|\\
&\leq ASt^2.
\end{align*} 

We estimate the final term using Lemma \ref{horizontalequation}. For $1\leq j<i\leq r$ and $t>0$ in $I$:
\begin{align*}
&|h_{ij}(t)-g_{ij}(t)| \\
&\leq \frac{1}{2} \int_{0}^{t}|h_{i}(s)||h_{j}'(s)-g_{j}'(s)|+|g_{j}'(s)||h_{i}(s)-g_{i}(s)|\\
&\qquad +\frac{1}{2} \int_{0}^{t}|h_{j}(s)||h_{i}'(s)-g_{i}'(s)|+|g_{i}'(s)||h_{j}(s)-g_{j}(s)|\\
& \leq AS \int_{0}^{t} s \dd s\\
& \leq AS t^2.
\end{align*}
The bound is the same for $t<0$ in $I$.

Combining our estimates of each term in \eqref{angleestimate} and using $0\leq A\leq 1$ gives 
\begin{align*}
d(g(t),h(t))\leq (C_K r+2C_KP\sqrt{S})\sqrt{A} |t| \qquad \mbox{for}\quad t\in I
\end{align*}
where $P=r(r-1)/2$. The conclusion with $C_{a}:=C_K r+2C_KP\sqrt{S}\geq 1$.
\end{proof}

The second geometric lemma is a generalization of \cite[Lemma 4.2]{PS16} in the Heisenberg group to the free Carnot group of step 2. The idea is to perturb a horizontal line to pass through a nearby point of interest. While the formulation may seem unusual, it is one of the key tools to prove that existence of an almost maximal directional derivative implies differentiability. Recall that $x\exp(tE)=x+tE(x)$ for $x\in \bbG_{r}$, $E\in V_{1}$ and $t\in \bbR$. Hence `horizontal lines' of the form $t\mapsto x+tE(x)$ are lines and horizontal curves in $\bbG_{r}$.

\begin{lemma}\label{curveforalmostmax}
Given $\eta>0$, there is $0<\Delta(\eta)<1/4$ and a constant $C_{m}\geq 1$ such that the following holds whenever $0<\Delta<\Delta(\eta)$. Suppose:
\begin{itemize}
\item $x, u\in \bbG_{r}$ with $d(u)\leq 1$,
\item $E\in V_{1}$ with $\omega(E)=1$,
\item $0<r<\Delta$ and $s:=r/\Delta$.
\end{itemize}
Then there is a Lipschitz horizontal curve $g\colon \bbR\to \bbG_{r}$ which is a concatenation of horizontal lines such that:
\begin{enumerate}
\item $g(t)=x+tE(x)$ for $|t|\geq s$,
\item $g(\zeta)=x\delta_{r}(u)$, where $\zeta:=r\langle u,E(0)\rangle$,
\item $\mathrm{Lip}_{\bbG}(g)\leq 1+\eta \Delta$,
\item $g'(t)$ exists and $|(p\circ g)'(t)-p(E)|\leq C_{m}\Delta$ for $t\in \bbR$ outside a finite set.
\end{enumerate}
\end{lemma}

\begin{proof}
Using group translations we can assume that $x=0$. Since $\bbG_{r}$ is a free Carnot group, we can also use Lemma \ref{existF} to assume $E=X_{1}$.

For $|t|\geq s$ the curve $g(t)$ is explicitly defined by (1) and satisfies (3) and (4). To define $g(t)$ for $|t|<s$ we consider the two cases $-s< t\leq \zeta$ and $\zeta \leq t < s$. These are similar so we show how to define the curve for $-s< t\leq \zeta$. By using left group translations by $\pm sE(0)$ and reparameterizing the curve, it suffices to show the following claim.
\end{proof}

\begin{claim}\label{claimforlater}
There exists $0<\Delta(\eta)<1/4$, depending on $\eta$ and $\bbG_{r}$, and there exists $C\geq 1$, depending only on $\bbG_{r}$, for which the following holds.

If $0<\Delta<\Delta(\eta)$ then there exists a Lipschitz horizontal curve $\varphi \colon [0, s+\zeta]\to \bbG_{r}$ such that $\varphi(0) = 0, \varphi(s+\zeta)=(s,0,\ldots,0)\delta_r(u)$ and:
\begin{itemize}
    \item[(A)] $\mathrm{Lip}_{\bbG}(\varphi)\leq 1+\eta\Delta$,
    \item[(B)] $\varphi'(t)$ exists and $|(p\circ\varphi)'(t)-(1,0,\ldots, 0)|\leq C \Delta$ for $t\in [0, s+\zeta]$ outside a finite set. 
\end{itemize}
\end{claim}

\begin{proof}[Proof of Claim]
Throughout the proof, $C\geq 1$ will denote a constant depending only on the Carnot group $\bbG_{r}$. Let $y=(s,0,\ldots,0) \delta_r(u)$. The definition of group product implies
\begin{equation}\label{coordhor}
y_1=s+ru_1,  \qquad \qquad y_i=ru_i \quad \mbox{for }i>1,
\end{equation}
\begin{equation}\label{coordver}
y_{i1}= r^2u_{i1}-\frac{1}{2}rs u_{i}\quad \mbox{for } i>1, \qquad \qquad y_{ij}= r^2 u_{ij} \quad \mbox{for }i>j>1.
\end{equation}

The main idea is to use Lemma \ref{curve} to construct the desired curve. However, Lemma \ref{curve} only gives curves joining $0$ to endpoints whose horizontal coordinates are of the form $(L,0,\ldots,0)$. To handle this, we first construct another curve $\beta$ which shifts the endpoint $y$ to one of the required form. The desired curve $\varphi$ will then be the join of $\beta$ with a curve $\alpha$ from Lemma \ref{curve}.

\medskip

\emph{Construction of the curve $\beta$.}
Define a horizontal curve $h\colon [0,1]\to \bbG_{r}$ as the horizontal lift (Definition \ref{horizontallift}) starting at $y$ of the straight line joining $p(y)$ to $(s/2, 0, \ldots, 0)$. Clearly
\[h_{1}(t)=(1-t)(s+ru_{1})+ts/2 \quad \mbox{and} \quad h_{i}(t)=(1-t)ru_{i} \quad \mbox{for }2\leq i\leq r.\]
Define $z=h(1)\in \bbG_{r}$. The horizontal coordinates of $z$ satisfy $p(z)=(s/2, 0, \ldots, 0)$, and a simple computation using \eqref{coordhor}, \eqref{coordver} and Definition \ref{horizontallift} yields the vertical coordinates
\[z_{i1}= r^2u_{i1} - \frac{1}{4}rsu_{i} \quad \mbox{if }i>1 \quad \mbox{and} \quad z_{ij}=r^2 u_{ij} \quad \mbox{if }i>j>1.\]
Equivalence of the CC metric and Koranyi quasi-metric implies that $|u_{ij}|\leq C$ for $i>j\geq 1$. It follows that
\begin{equation}\label{ABbounds}
|z_{i1}|\leq Crs \quad \mbox{if }i>1 \quad \mbox{and} \quad |z_{ij}|\leq Cr^{2} \quad \mbox{if }i>j>1.
\end{equation}
Clearly
\begin{equation}\label{derivh}
(p\circ h)'(t)=(-s/2-ru_{1}, -ru_{2}, \ldots, -ru_{r}).
\end{equation}
Since $ru_{1}=\zeta$, we obtain
\[|(p\circ h)'+ (s/2+\zeta, 0, \ldots, 0)|=|(0,ru_{2}, \ldots, ru_{r})| \leq r. \]

Now define $\beta\colon [s/2, s+\zeta]\to \bbG_{r}$ joining $z$ to $y$ by $\beta(t)=h\Big(1 - \frac{(t-s/2)}{(s/2+\zeta)} \Big)$. Assuming $\Delta<1/4$ gives
\begin{equation}\label{april5}
\frac{r}{s/2 + ru_{1}}\leq \frac{4r}{s}=4\Delta.
\end{equation}
Hence
\[|(p\circ \beta)'-(1,0,\ldots,0)|\leq \frac{r}{s/2 + \zeta}\leq 4\Delta.\]
Using \eqref{derivh}, \eqref{april5} and $d(u)\leq 1$ gives
\begin{align*}
\mathrm{Lip}(p\circ \beta) &\leq \frac{|(s/2+ru_{1},ru_{2},\ldots, ru_{r})|}{(s/2+\zeta)}\\
&=\Big( 1+\frac{r^2u_2^2}{(s/2+ru_1)^2}+\ldots+\frac{r^2u_r^2}{(s/2+ru_1)^2}\Big)^{1/2}\\
&\leq \Big(1+16\Delta^2(u_2^2+\ldots + u_r^2) \Big)^{1/2}\\
&\leq 1+16\Delta^2\\
&\leq 1+\eta\Delta,
\end{align*}
assuming in the last inequality that $\Delta \leq \eta/16$.

\medskip

\emph{Construction of the curve $\alpha$.}
Let $A:=\max_{2\leq i\leq r}|z_{i1}|$ and $B:=\max_{i>j>1}|z_{ij}|$. Since $p(z)=(s/2, 0, \ldots, 0)$, Lemma \ref{curve} gives the existence of a Lipschitz horizontal curve $g\colon [0,1]\to \bbG _{r}$ joining $0$ to $z$ such that
\begin{enumerate}
\item[(A')] $\mathrm{Lip}_{\bbG}(g)\leq \max \left\{ (s/2)\Big(1+\frac{256P^4A^2}{s^4}\Big)^{1/2}, \, (s/2)\Big(1+\frac{1152P^2B}{s^2}\Big)^{1/2} \right\}$,
\item[(B')] $g'(t)$ exists for all $t\in [0,1]$ except finitely many points and satisfies
\[ |(p\circ g)'(t) - p(z)|\leq \max \left\{ \frac{8P^2A}{s}, \, 24P\sqrt{B} \right\}.\]
\end{enumerate}
Recall from \eqref{ABbounds} that $A\leq Crs$ and $B\leq Cr^2$. It easily follows from (A') that
\[\mathrm{Lip}_{\bbG}(g)\leq (s/2)+C\Delta^{2}s \leq (s/2)+\eta\Delta s/2\]
using in the second inequality that $\Delta$ is sufficiently small. Similarly, it follows from (B') that
\[|(p\circ g)' - (s/2,0,\ldots,0)|\leq C\Delta s.\]

Now define $\alpha\colon [0,s/2] \to \bbG_{r}$ by $\alpha(t)=g(2t/s)$. The curve $\alpha$ satisfies:
\[\mathrm{Lip}_{\bbG}(\alpha) \leq 1+\eta \Delta,\]
\[|(p\circ \alpha)' - (1,0,\ldots,0)|\leq C \Delta.\]
Joining $\alpha$ and $\beta$ gives a curve $\varphi$ with the desired properties.
\end{proof}

The following result is the analogue of \cite[Lemma 1.1]{DM12} and \cite[Lemma 2.1]{DMf} in the linear setting. 

\begin{lemma}\label{uds}
Let $L$ be the union of images of all curves constructed in Lemma \ref{curveforalmostmax}, where $x, u \in \bbQ^{n}$, $E$ is a linear combination of $\{X_{i}: 1\leq i\leq r\}$ with rational coefficients, and $r,s \in \bbQ$. Then there is a $G_{\delta}$ set $N\subset \bbG_{r}$ containing $L$ of CC-Hausdorff dimension one.
\end{lemma}

\begin{proof}
It suffices to prove that the desired set $N$ can be constructed with CC-Hausdorff dimension less than or equal to one. For any $S\subseteq \bbG_{r}$ and $r> 0$, let
\[
B_r(S):=\left\{x\in\mathbb{G}_{r}\ : \inf_{y\in S}d(x,y)<r\right\}.
\]
Let $I$ be a horizontal line segment of length at most one, $r,\delta>0$ and $k\geq 4/\delta$ a positive integer. Since $I$ is horizontal, we can cover $B_{1/k}(I)$ with $k$ open balls with diameter $4/k\leq \delta$. Hence $\mathcal{H}^r_{\delta}(B_{1/k}(I))\leq 4^r k^{1-r}$.

The set $L$ contains countably many curves, each of which is either a line or a union of finitely many lines. We can write $L=\bigcup_{i=1}^{\infty} L_i$, where $L_i$ is a horizontal line segment of length less than or equal to one. As in \cite[Lemma 1.1]{DM12}, we define
\[
 N_i:=\bigcup_{j=1}^{\infty} B_{1/2^{i+j}}(L_j), \qquad N:=\bigcap_{i=1}^{\infty} N_i.
\]
Then $N$ is a $G_{\delta}$ set containing $L$. An easy computation, as in \cite[Lemma 1.1]{DM12}, shows that $\mathcal{H}^{r}_{\delta}(N_{i}) \to 0$ as $i\to \infty$ for every $r>1$. Hence the CC-Hausdorff dimension of $N$ is less than or equal to one.
\end{proof}

The following theorem asserts that the existence of an `almost maximal' directional derivative in $N$ suffices for differentiability.

\begin{notation}\label{D^f}
Fix a Lebesgue null $G_{\delta}$ set $N\subset \bbG_{r}$ as in Lemma \ref{uds}. For any Lipschitz function $f:\bbG_{r}\to \bbR$, define:
\[D^{f}:=\{ (x,E) \in N\times V_{1}\colon \omega(E)=1,\, Ef(x) \mbox{ exists}\}.\]
\end{notation}

\begin{proposition}\label{almostmaximalityimpliesdifferentiability}
Let $f\colon \bbG_{r} \to \mathbb{R}$ be a Lipschitz function with $\mathrm{Lip}_{\bbG}(f) \leq 1/2$. Suppose $(x_{\ast}, E_{\ast})\in D^{f}$. Let $M$ denote the set of pairs $(x,E)\in D^{f}$ such that $Ef(x)\geq E_{\ast}f(x_{\ast})$ and
\begin{align*}
& |(f(x + tE_{\ast}(x))-f(x)) - (f(x_{\ast} + tE_{\ast}(x_{\ast}))-f(x_{\ast}))| \\
& \qquad \leq 6|t| (  (Ef(x)-E_{\ast}f(x_{\ast}))\mathrm{Lip}_{\bbG}(f))^{\frac{1}{4}}
\end{align*}
for every $t\in (-1,1)$. If
\[\lim_{\delta \downarrow 0} \sup \{Ef(x)\colon (x,E)\in M \mbox{ and }d(x,x_{\ast})\leq \delta\}\leq E_{\ast}f(x_{\ast})\]
then $f$ is Pansu differentiable at $x_{\ast}$ with Pansu differential
\[L(x)=E_{\ast}f(x_{\ast})\langle x , E_{\ast}(0) \rangle=E_{\ast}f(x_{\ast})\langle p(x) , p(E_{\ast}) \rangle.\]
\end{proposition}

To prove Proposition \ref{almostmaximalityimpliesdifferentiability}, the idea is to argue by contradiction. We will first use Lemma \ref{curveforalmostmax} to modify the line $x_{\ast} + tE_{\ast}(x_{\ast})$ to form a Lipschitz horizontal curve $g$ in $N$ which passes through a nearby point showing non-Pansu differentiability at $x_{\ast}$. One then applies \cite[Lemma 3.4]{Pre90} with $\varphi=f\circ g$ to obtain a large directional derivative along $g$ and estimates for difference quotients in the new direction. One then develops these estimates to show that the new point and direction form a pair in $M$. This shows that there is a nearby pair in $M$ giving a larger directional derivative than $(x_{\ast},E_{\ast})$, a contradiction.

The proof of Proposition \ref{almostmaximalityimpliesdifferentiability} is almost the same as the proof of \cite[Theorem 5.6]{PS16}, replacing \cite[Lemma 4.2]{PS16} and \cite[Lemma 5.1]{PS16} in the Heisenberg group with Lemma \ref{curveforalmostmax} and Lemma \ref{closedirectioncloseposition} in the step 2 free Carnot group from the present paper. Note that the geometry of the individual Carnot group is strongly used in the proof when constructing two horizontal curves. We show how to begin the proof and construct these curves, then refer the reader to \cite[Theorem 5.6]{PS16} for the later parts which are the same as those treated there.

\begin{proof}[Proof of Proposition \ref{almostmaximalityimpliesdifferentiability}]
We can assume $\mathrm{Lip}_{\mathbb{H}}(f)>0$ since otherwise the statement is trivial. Let $\varepsilon>0$ and fix various parameters as follows.

\vspace{0.2cm}

\emph{Parameters.} Choose:
\begin{enumerate}
\item $0< v<1/32$ such that $4(1+20v)\sqrt{(2+v)/(1-v)}+v < 6$,
\item $\eta=\varepsilon v^{3}/3200$,
\item $\Delta(\eta/2)$, $C_{\mathrm{m}}$ and $C_{\mathrm{a}}=C_{\mathrm{angle}}(2)$ using Lemma \ref{curveforalmostmax} and Lemma \ref{closedirectioncloseposition}
\item rational $0< \Delta < \min \{\eta v^2,\, \Delta(\eta/2),\, \varepsilon v^{5}/(8C_{\mathrm{m}}^{2}C_{\mathrm{a}}^{4}\mathrm{Lip}_{\mathbb{G}}(f)^3) \}$,
\item $\sigma=9\varepsilon^{2}v^{5}\Delta^2/256$,
\item $0<\rho<1$ such that
\begin{equation}\label{directionaldifferentiability}
|f(x_{\ast} + tE_{\ast}(x_{\ast})) - f(x_{\ast})-tE_{\ast}f(x_{\ast})|\leq \sigma \mathrm{Lip}_{\mathbb{G}}(f)|t|
\end{equation}
for every $|t|\leq \rho$,
\item $0<\delta < \rho \sqrt{3\varepsilon v\Delta^{3}}/4$ such that
\[Ef(x)<E_{\ast}f(x_{\ast})+\varepsilon v\Delta/2\]
whenever $(x,E)\in M$ and $d(x,x_{\ast})\leq 4\delta(1+1/\Delta)$.
\end{enumerate}
To prove Pansu differentiability of $f$ at $x_{\ast}$ we will show:
\[|f(x_{\ast}\delta_{t}(h))-f(x_{\ast})-tE_{\ast}f(x_{\ast})\langle h, E_{\ast}(0) \rangle |\leq \varepsilon t\]
whenever $d(h) \leq 1$ and $0<t<\delta$. Suppose this is not true. Then there exists $u\in \mathbb{Q}^{n}$ with $d(u) \leq 1$ and rational $0<r<\delta$ such that:
\begin{equation}\label{badpoint}
|f(x_{\ast}\delta_{r}(u))-f(x_{\ast})-rE_{\ast}f(x_{\ast})\langle u, E_{\ast}(0) \rangle|> \varepsilon r.
\end{equation}
Let $s=r/ \Delta \in \mathbb{Q}$. We next construct Lipschitz horizontal curves $g$ and $h$ for which we can apply the mean value type lemma \cite[Lemma 3.4]{Pre90} with $\varphi:=f\circ g$ and $\psi:=f\circ h$.

\medskip

\emph{Construction of $g$.} To ensure that the image of $g$ is a subset of the set $N$, we first introduce rational approximations to $x_{\ast}$ and $E_{\ast}$. 

Since the Carnot-Carath\'{e}odory and Euclidean distances are topologically equivalent, $\mathbb{Q}^{n}$ is dense in $\mathbb{R}^{n}$ with respect to the distance $d$. The set
\[ \{E\in V\colon \omega(E)=1, \, E \mbox{ a rational linear combination of }X_{i},\, 1\leq i\leq r\}\]
is dense in $\{E\in V\colon \omega(E)=1\}$ with respect to the norm $\omega$. To see this, suppose $E\in V$ satisfies $\omega(E)=1$. The Euclidean sphere $\mathbb{S}^{r-1} \subset \mathbb{R}^{r}$ contains a dense set $S$ of points with rational coordinates. This fact is well known, e.g. one can use stereographic projection. Let $q=(q_1,\ldots, q_{r})\in \mathbb{S}^{r-1}$ be the coefficients of $E$ in the basis $\{X_{i}: 1\leq i\leq r\}$. Take $\widetilde{q}\in S$ such that $| q-\widetilde{q}|$ is small and define the rational approximation of $E$ as the linear combination of $\{X_{i}: 1\leq i\leq r\}$ with coefficients $\widetilde{q}_i$. 

Define
\begin{equation}\label{A1}A_{1}=(\eta \Delta/C_{\mathrm{a}})^{2}\end{equation}
and
\begin{equation}\label{A2}A_{2}=\Big(6- \Big( 4(1+20v) \Big( \frac{2+v}{1-v} \Big)^{\frac{1}{2}}+v\Big)\Big)^{2} \frac{(  \varepsilon v\Delta/2   )\mathrm{Lip}_{\mathbb{G}}(f))^{\frac{1}{2}}}{C_{\mathrm{a}}^{2}\mathrm{Lip}_{\mathbb{G}}(f)^{2}}.\end{equation}
Notice $A_{1}, A_{2}>0$ using, in particular, our choice of $v$. Choose $\widetilde{x}_{\ast}\in \mathbb{Q}^{n}$ and $\widetilde{E}_{\ast}\in V$ with $\omega( \widetilde{E}_{\ast})=1$, a rational linear combination of $\{X_{i}: 1\leq i\leq r\}$, sufficiently close to $x_{\ast}$ and $E_{\ast}$ to ensure:
\begin{equation}\label{nowlistingthese} d(\widetilde{x}_{\ast}\delta_{r}(u),x_{\ast})\leq 2r,\end{equation}
\begin{equation}\label{lista} d(\widetilde{x}_{\ast}\delta_{r}(u), x_{\ast}\delta_{r}(u))\leq \sigma r,\end{equation}
\begin{equation}\label{listvector1} \omega(\widetilde{E}_{\ast}-E_{\ast})\leq \min \{ \sigma, \, C_{\mathrm{m}}\Delta,\, A_{1},\, A_{2} \},\end{equation}
which is possible since all terms on the right side of the above inequalities are strictly positive.

Note $0<r<\Delta$ and recall $s=r/\Delta$ is rational. To construct $g$ we apply Lemma \ref{curveforalmostmax} with the following parameters:
\begin{itemize}
\item $\eta, r, \Delta$ and $u$ as defined above in \eqref{badpoint},
\item $x=\widetilde{x}_{\ast}$ and $E=\widetilde{E}_{\ast}$.
\end{itemize}
This gives a Lipschitz horizontal curve $g\colon \mathbb{R} \to \mathbb{G}_r$ which is a concatenation of horizontal lines such that:
\begin{enumerate}
\item $g(t)=\widetilde{x}_{\ast} + t\widetilde{E}_{\ast}(x_{\ast})$ for $|t|\geq s$,
\item $g(\zeta)=\widetilde{x}_{\ast}\delta_{r}(u)$, where $\zeta := r\langle u,\widetilde{E}_{\ast}(0)\rangle$,
\item $\mathrm{Lip}_{\mathbb{G}}(g)\leq 1+\eta\Delta$,
\item $g'(t)$ exists and $|(p\circ g)'(t) - p(\widetilde{E}_{\ast})| \leq C_{\mathrm{m}}\Delta$ for $t\in \mathbb{R}$ outside a finite set.
\end{enumerate}
Since the relevant quantities were rational and the set $N$ was chosen using Lemma \ref{uds}, we also know that the image of $g$ is contained in $N$.

\medskip

\emph{Construction of $h$.} 

\begin{claim}
There exists a Lipschitz horizontal curve $h\colon \mathbb{R}\to \mathbb{G}_r$ such that:
\[h(t)= \begin{cases} \widetilde{x}_{\ast} + t\widetilde{E}_{\ast}(\widetilde{x}_{\ast}) &\mbox{if }|t|\geq s,\\
x_{\ast} + tE_{\ast}(x_{\ast}) &\mbox{if }|t|\leq s/2.
\end{cases}\]
Also, $h$ satisfies the estimates:
\begin{itemize}
\item $\mathrm{Lip}_{\mathbb{G}}(h)\leq 1+\eta\Delta/2,$
\item the derivative $h'(t)$ exists for all but finitely many $t$ and satisfies the bound $|(p\circ h)'(t)-p(E_{\ast})|\leq \min \{A_{1}, A_{2}\}$.
\end{itemize}
\end{claim}

\begin{proof}[Proof of Claim]
Using left translations and Lemma \ref{existF}, we may start by assuming that $x_{\ast}=0$ and $E_{\ast}=X_1$. Clearly $h(t)$ is defined explicitly and satisfies the required conditions for $|t|\leq s/2$ and $|t|\geq s$. We now show how to construct $h(t)$ for $s/2<t<s$. The case $-s<t<-s/2$ is essentially identical.

Recall $\Delta(1)$ and $C$ from Claim \ref{claimforlater}. Choose $0<\Gamma<\Delta(1)$ which satisfies
\[(1+\Gamma)^{2}\leq 1+\eta\Delta/2\]
and
\[C\Gamma(1+\Gamma)+\Gamma \leq \max \{A_{1}, A_{2}\}.\]
Define $\lambda=s\Gamma/2<\Gamma$. Choose $v\in \bbG_{r}$ with $d(v)\leq 1$ and $|v_{i}|\leq 1$ for $i=1, \ldots, n$, such that
\[\delta_{\lambda}(v)=(-s,0,\ldots, 0)(\widetilde{x_{\ast}}+s\widetilde{E}_{\ast}(\widetilde{x}_{\ast})).\]
This is possible if the rational approximation introduced earlier is sufficiently good; note that the rational approximation was introduced after all quantities upon which $\lambda$ depends. We now apply Claim \ref{claimforlater} with
\begin{itemize}
\item $\eta=1$ and $\Delta$ replaced by $\Gamma$,
\item $r$ replaced by $\lambda$ and $s$ replaced by $s/2$,
\item $u$ replaced by $v$,
\item $\zeta$ replaced by $\lambda v_{1}$.
\end{itemize}
We obtain a Lipschitz horizontal curve $\varphi: [0,s/2+\lambda v_{1}]\to \mathbb{G}_r$ such that:
\begin{itemize}
\item $\varphi(0)=0$,
\item $\varphi(s/2+\lambda v_{1})=(-s/2, 0, \ldots, 0)(\widetilde{x_{\ast}}+s\widetilde{E}_{\ast}(\widetilde{x}_{\ast}))$,
\item $\mathrm{Lip}_{\bbG}(\varphi)\leq 1+\Gamma$,
\item $\varphi'(t)$ exists and $|(p\circ \varphi)'(t)-(1,0,\ldots, 0)|\leq C\Gamma$ for all but finitely many $t\in [0,s/2+\lambda v_{1}]$.
\end{itemize}
Then $\widetilde{\varphi}:[0,1]\to \mathbb{G}_r$ defined by $\widetilde{\varphi}(t)=\varphi((s/2+\lambda v_{1})t)$ is a Lipschitz horizontal curve such that:
\begin{itemize}
\item $\widetilde{\varphi}(0)=0$,
\item $\widetilde{\varphi}(1)=(-s/2, 0, \ldots, 0)(\widetilde{x_{\ast}}+s\widetilde{E}_{\ast}(\widetilde{x}_{\ast}))$,
\item $\mathrm{Lip}_{\bbG}(\widetilde{\varphi})\leq (1+\Gamma)(s/2+\lambda v_{1})$,
\item $\widetilde{\varphi}'(t)$ exists and $|(p\circ \widetilde{\varphi})'(t)-(s/2+\lambda v_{1}, 0, \ldots, 0)|\leq C\Gamma(s/2+\lambda v_{1})$ for all but finitely many $t\in [0,1]$.
\end{itemize}
Define $h_1 :[s/2,s]\to \mathbb{G}_r$ by:
\[h_1(t)=(s/2, 0, \ldots, 0) \widetilde{\varphi}((2/s)(t-s/2)).\]
Then $h_1$ is a Lipschitz horizontal curve which satisfies $h_{1}(s/2)=(s/2, 0, \ldots, 0)$ and $h_{1}(s)=\widetilde{x_{\ast}}+s\widetilde{E}_{\ast}(\widetilde{x}_{\ast})$. Further:
\begin{align*}
\mathrm{Lip}_{\bbG}(h_1)&\leq \frac{2(1+\Gamma)(s/2+\lambda v_{1})}{s}\\
& \leq \frac{2(1+\Gamma)(s/2+\lambda)}{s}\\
&\leq (1+\Gamma)^{2}\\
&\leq 1+\eta\Delta/2.
\end{align*}
Clearly, for all but finitely many $t\in [s/2,s]$:
\[|(p\circ h_1)'(t)-(1+2\lambda v_{1}/s, 0, \ldots, 0)| \leq C\Gamma (1+2\lambda v_{1}/s).\]
This implies:
\begin{align*}
|(p\circ h_1)'(t)-(1, 0, \ldots, 0)| &\leq C\Gamma(1+2\lambda v_{1}/s) + 2|\lambda v_{1}|/s\\
&\leq C\Gamma(1+\Gamma)+\Gamma\\
&\leq \max \{A_{1}, A_{2}\}.
\end{align*}

The curve $h_{1}(t)$ defines $h(t)$ for $t \in [s/2,s]$ with the desired properties. The construction of $h(t)$ for $t\in [-s,-s/2]$ is essentially identical.
\end{proof}

Using the constructed curves $g$ and $h$, together with Lemma \ref{closedirectioncloseposition}, the rest of the proof is identical to the Heisenberg group case treated in \cite[Theorem 5.6]{PS16}.
\end{proof}

Proposition \ref{DoreMaleva} below is a generalization of \cite[Proposition 6.1]{PS16} in the Heisenberg group, itself based on \cite[Theorem 3.1]{DM11} in Euclidean spaces, to free Carnot groups of step 2. We omit the proof since it is identical to \cite[Proposition 6.1]{PS16}, replacing an application of \cite[Lemma 5.1]{PS16} in the Heisenberg group with Lemma \ref{closedirectioncloseposition} in the step 2 free Carnot group.

\begin{proposition}\label{DoreMaleva}
Suppose $f_0:\bbG_{r}\to \bbR$ is a Lipschitz function, $(x_0,E_0)\in D^{f_0}$ and $\delta_0, \mu, K>0$. Then there is a Lipschitz function $f:\bbG_{r}\to \bbR$ such that $f-f_0$ is $\bbG$-linear with $\mathrm{Lip}_{\bbG}(f-f_{0})\leq \mu$, and a pair $(x_{\ast},E_{\ast})\in D^{f}$ with $d(x_{\ast},x_0)\leq \delta_0$ such that $E_{\ast}f(x_{\ast})>0$ is almost locally maximal in the following sense.

For any $\varepsilon>0$ there is $\delta_{\varepsilon}>0$ such that whenever $(x,E)\in D^{f}$ satisfies both:
\begin{enumerate}
\item $d(x,x_{\ast})\leq \delta_{\varepsilon}$, $Ef(x)\geq E_{\ast}f(x_{\ast})$,
\item for any $t\in (-1,1)$:
\begin{align*}
&|(f(x+tE_{\ast}(x))-f(x))-(f(x_{\ast}+tE_{\ast}(x_{\ast}))-f(x_{\ast}))|\\
& \qquad \leq K|t| ( Ef(x)-E_{\ast}f(x_{\ast}))^{\frac{1}{4}},
\end{align*}
\end{enumerate}
then:
\[Ef(x)<E_{\ast}f(x_{\ast})+\varepsilon.\]
\end{proposition}

Proposition \ref{DoreMaleva} shows that for any Lipschitz function $f_{0}\colon \bbG_{r} \to \bbR$ there is a Lipschitz function $f\colon \bbG_{r} \to \bbR$ such that $f-f_{0}$ is $\bbG$-linear and $f$ has an almost maximal directional derivative in the sense of Proposition \ref{almostmaximalityimpliesdifferentiability}. Notice that if $f-f_{0}$ is $\bbG$-linear then $D^{f}=D^{f_{0}}$ and the functions $f, f_{0}$ have the same points of Pansu differentiability. Since Proposition \ref{almostmaximalityimpliesdifferentiability} asserts that an almost maximal directional derivative suffices to differentiability, we deduce that any Lipschitz function $f_{0}\colon \bbG_{r}\to \bbR$ is Pansu differentiable at a point of $N$.

\begin{theorem}\label{step2freeUDS}
Any free Carnot group of step 2 contains a UDS of CC-Hausdorff dimension one.
\end{theorem}

\section{Generalization to more general Carnot groups}\label{morecarnotgroups}

In this section we see how to generalize results proved for free Carnot groups of step 2 to general Carnot groups of step 2.

To do this we suppose $\bbG$ and $\bbH$ are Carnot groups with horizontal layers $V$ and $W$ for which there exist bases $\mathbf{X}=(X_{1}, \ldots, X_{r})$ and $\mathbf{Y}=(Y_{1}, \ldots, Y_{r})$, together with a Lie group homomorphism $F\colon \bbG \to \bbH$ such that $F_{*}(X_{i})=Y_{i}$ for $1\leq i\leq r$. 

Note that if $\bbH$ is any Carnot group of rank $r$ and step $s$, then one possibility is to let $\bbG$ be the free Carnot group of the same rank and step. Indeed, suppose $X_{1}, \ldots, X_{r}$ and $Y_{1}, \ldots, Y_{r}$ are bases of the horizontal layers of $\bbG$ and $\bbH$ respectively. Since the Lie algebra of $\bbG$ is free-nilpotent (Definition \ref{freeliealgebra}), there exists a Lie algebra homomorphism $\Phi$ such that $\Phi(X_{i})=Y_{i}$. Since Carnot groups are simply connected, there exists a Lie group homomorphism $F\colon \bbG \to \bbH$ such that $dF=\Phi$, which has the properties required.

Equip $\bbG$ and $\bbH$ with the CC metrics $d_{\bbG}$ and $d_{\bbH}$ induced by the bases $\mathbf{X}$ and $\mathbf{Y}$ respectively. The map $F$ preserves lengths of horizontal curves, so is Lipschitz with $\mathrm{Lip}(F)= 1$. 

Let $p_{\bbG}\colon \bbG\to \bbR^{r}$ and $p_{\bbH}\colon \bbH \to \bbR^{r}$ denote projections onto the first $r$ coordinates. It follows from the definition of $F$ that if $u=(u_{h},0)\in \bbG$ for some $u_{h}\in \bbR^{r}$ then $F(u)=(u_{h},0)\in \bbH$ (note though that $\bbG$ and $\bbH$ may be represented in coordinates by Euclidean spaces of different dimensions). Also, $p_{\bbH}(F(u))=p_{\bbG}(u)$ for any $u\in \bbG$.\\

The following lemma is well-known, but we include its short proof for
completeness.

\begin{lemma}\label{lift}
Let $x\in \bbG$ and $\tilde{x}=F(x)\in \bbH$. Then for any $\tilde{y}\in \bbH$, there is $y\in \bbG$ with $F(y)=\tilde{y}$ and $d_{\bbG}(x,y)=d_{\bbH}(\tilde{x}, \tilde{y})$.
\end{lemma}

\begin{proof}
Fix a horizontal curve $\gamma\colon [0,1]\to \bbH$ joining $\tilde{x}$ to $\tilde{y}$ with $L(\gamma)=d_{\bbH}(\tilde{x},\tilde{y})$. Choose integrable controls $a_{i}\in L^{1}([0,1])$ such that
\[ \gamma'(t)=\sum_{i=1}^{r}a_{i}(t)Y_{i}(\gamma(t)) \qquad \mbox{for a.e. }t\in [0,1].\] 
Define a Lipschitz horizontal curve $\phi \colon [0,1]\to \bbG$ by
\[\phi(0)=x \qquad \mbox{and}\qquad \phi'(t)=\sum_{i=1}^{r}a_{i}(t)X_{i}(\phi(t))\qquad \mbox{for a.e. }t\in [0,1].\] 
Clearly $L(\phi)=L(\gamma)$ because $\mathbf{X}$ and $\mathbf{Y}$ are orthonormal bases of $V$ and $W$. Let $y:=\phi(1)$. Since $F_{*}(X_{i})=Y_{i}$, we have $F\circ \phi=\gamma$ and hence $F(y)=\tilde{y}$. Since $\phi$ joins $x$ to $y$, we have
\[d_{\bbG}(x,y)\leq L(\phi)=L(\gamma)=d_{\bbH}(\tilde{x},\tilde{y}).\]
Now take a horizontal curve $\psi\colon [0,1]\to \bbG$ joining $x$ to $y$ with $L(\psi)=d_{\bbG}(x,y)$. Then $F\circ \psi \colon [0,1]\to \bbH$ is a horizontal curve joining $\tilde{x}$ to $\tilde{y}$ and $L(F\circ \psi)=L(\psi)$. Hence
\[d_{\bbH}(\tilde{x},\tilde{y})\leq L(F\circ \psi)=L(\psi)=d_{\bbG}(x,y).\]
We conclude that $d_{\bbH}(\tilde{x},\tilde{y})=d_{\bbG}(x,y)$.
\end{proof}

\begin{proposition}\label{quotientdiffCC}
If the CC distance in $\bbG$ is differentiable in horizontal directions then the CC distance in $\bbH$ is differentiable in horizontal directions.
\end{proposition}

\begin{proof}
Let $\tilde{u}=(u_{h},0)\in \bbH$ for some $u_{h}\in \bbR^{r}$. Let $u=(u_{h},0)\in \bbG$, which satisfies $F(u)=\tilde{u}$ and $d_{\bbG}(u)=d_{\bbH}(\tilde{u})$. Given $\tilde{z}\in \bbH$, use Lemma \ref{lift} to find $z\in \bbG$ such that $F(z)=\tilde{z}$ and $d_{\bbG}(z)=d_{\bbH}(\tilde{z})$. Since $F$ is a group homomorphism and $\mathrm{Lip}(F)=1$, using Remark \ref{derivativeformula} we have:
\begin{align*}
d_{\bbH}(\tilde{u}\tilde{z})\leq d_{\bbG}(uz)&\leq d_{\bbG}(u)+\langle u_{h}/d_{\bbG}(u),p_{\bbG}(z)\rangle + o(d_{\bbG}(z))\\
&=d_{\bbH}(\tilde{u})+\langle u_{h}/d_{\bbH}(\tilde{u}),p_{\bbH}(\tilde{z})\rangle+o(d_{\bbH}(\tilde{z})).
\end{align*}
The opposite inequality is proved in Lemma \ref{distanceinequality}, hence $d_{\bbH}$ is differentiable at $u$.
\end{proof}

\begin{proof}[Proof of Theorem \ref{whatgroupsdiff}:]
Combining Theorem \ref{thmdifferentiabilityofdistance} and Proposition \ref{quotientdiffCC} (with $\bbH$ any given step 2 Carnot group and $\bbG$ the step 2 free Carnot group with horizontal layer of the same dimension as that of $\bbH$), we deduce that the CC distance is differentiable in horizontal directions in any step 2 Carnot group. The second part of the result is Theorem \ref{Conterex:Engel}.
\end{proof}
\begin{proposition}\label{quotientUDS}
If $\bbG$ admits a (CC-Hausdorff dimension one) universal differentiability set $N$, then the image $F(N)$ is a (CC-Hausdorff dimension one) universal differentiability set in $\bbH$.
\end{proposition}

\begin{proof}
Let $N$ be a universal differentiability set in $\bbG$. Clearly if $N$ has CC-Hausdorff dimension one then $F(N)$ has CC-Hausdorff dimension at most one, since $F$ is Lipschitz. We show that $F(N)$ is a UDS, from which it will follow (by Remark \ref{Hausone}) that $F(N)$ has CC-Hausdorff dimension exactly one. Let $f\colon \bbH\to \bbR$ be a Lipschitz function. Then $f\circ F\colon \bbG\to \bbR$ is a Lipschitz function on $\bbG$. Since $N$ is a UDS in $\bbG$, $f\circ F$ is Pansu differentiable at a point $x\in N$. We will show $f$ is Pansu differentiable at $\tilde{x}:=F(x)\in F(N)$.

Let $L_{\bbG}$ be the Pansu differential of $f\circ F$ at $x$. Choose $v\in \bbR^{r}$ such that $L_{\bbG}(z)=\langle v,p_{\bbG}(z) \rangle$. Fix $\varepsilon>0$. Then there exists $\delta>0$ such that if $d_{\bbG}(y,x)<\delta$ then
\begin{equation}\label{diffindomain}
|f(F(y)) - f(F(x)) - \langle v,p_{\bbG}(x^{-1}y) \rangle| \leq \varepsilon d_{\bbG}(x,y).
\end{equation}

Now suppose $\tilde{y}\in \bbH$ with $d_{\bbH}(\tilde{y}, \tilde{x})<\delta$. Choose $y\in \bbG$ with $F(y)=\widetilde{y}$ and $d_{\bbG}(x,y)=d_{\bbH}(\tilde{x}, \tilde{y})<\delta$. Since $F$ acts as the identity on the first $r$ coordinates, we have
\[ p_{\bbH}(\tilde{x}^{-1}\tilde{y}) = p_{\bbH}(F(x)^{-1}F(y)) = p_{\bbH}(F(x^{-1}y)) = p_{\bbG}(x^{-1}y).\]
Hence \eqref{diffindomain} implies
\[ |f(\tilde{y}) - f(\tilde{x}) - \langle v,p_{\bbH}(\tilde{x}^{-1}\tilde{y}) \rangle| \leq \varepsilon d_{\bbH}(\tilde{x},\tilde{y}).\]
This shows $f$ is differentiable at $\tilde{x}\in F(N)$ with Pansu differential $L_{\bbH}(z) = \langle v,p_{\bbH}(z)\rangle$. It follows that $N$ is a UDS in $\bbH$. 
\end{proof}

\begin{proof}[Proof of Theorem \ref{maintheorem}:]
Combining Theorem \ref{step2freeUDS} and Proposition \ref{quotientUDS} gives the existence of a universal differentiability set of CC-Hausdorff dimension one in any step 2 Carnot group.
\end{proof}


\begin{thebibliography}{99}
\bibitem{ABCK97} Agrachev A. A., Bonnard B., Chyba M., Kupka I.: \emph{Sub-Riemannian sphere in Martinet flat case}, ESAIM Control Optim. Calc. Var. 2 (1997), 377--448.
\bibitem{AM14} Alberti, G., Marchese, A.: \emph{On the differentiability of Lipschitz functions with respect to measures in Euclidean spaces}, Geom. Funct. Anal. 26(1) (2016), 1--66.
\bibitem{ACP10} Alberti, G., Csornyei, M., Preiss, D.: \emph{Differentiability of Lipschitz functions, structure of null sets, and other problems}, Proc. Int. Congress Math. III (2010), 1379--1394.
\bibitem{AS15} Ardentov, A. A., Sachkov, Yu. L.: \emph{Cut time in sub-Riemannian problem on Engel group}, ESAIM Control Optim. Calc. Var. 21(4) (2015),  958--988.
\bibitem{Bat15} Bate, D.: \emph{Structure of measures in Lipschitz differentiability spaces}, J. Amer. Math. Soc. 28 (2015), 421--482.
\bibitem{Ber} Berestovski{\u\i}, V. N.: \emph{Homogeneous manifolds with an intrinsic metric. {I}}, Sibirsk. Mat. Zh. 29(6) (1988), 17--29.
\bibitem{BLU07} Bonfiglioli, A., Lanconelli, E., Uguzzoni, F.: \emph{Stratified Lie Groups and Potential Theory for Their Sub-Laplacians}, Springer Monographs in Mathematics (2007).
\bibitem{Burago-Burago-Ivanov} Burago, D., Burago, Y., Ivanov, S.: \emph{A course in metric geometry}, Graduate Studies in Mathematics 33, American Mathematical Society, Providence, RI (2001).
\bibitem{Cas14} Serra Cassano, F.: \emph{Some topics of geometric measure theory in Carnot groups}, Geometry, analysis and dynamics on sub-Riemannian manifolds 1, 1--121, EMS Ser. Lect. Math., Eur. Math. Soc., Zurich (2016). 
\bibitem{Che99} Cheeger, J.: \emph{Differentiability of Lipschitz functions on metric measure spaces}, Geom. Funct. Anal. 9(3) (1999), 428--517.
\bibitem{CDPT07} Capogna, L., Danielli, D., Pauls, S., Tyson, J.: \emph{An introduction to the Heisenberg group and the sub-Riemannian isoperimetric problem}, Birkhauser Progress in Mathematics 259 (2007).
\bibitem{CJ15} Csornyei, M., Jones, P.: \emph{Product formulas for measures and applications to analysis and geometry}, announced result: http://www.math.sunysb.edu/Videos/dfest/PDFs/38-Jones.pdf.
\bibitem{DhR} De Philippis, G., Rindler, F.: \emph{On the structure of A-free measures and applications}, Ann. of Math. 184 (2016), 1017-1039.
\bibitem{DM11} Dor\'e, M., Maleva, O.: \emph{A compact null set containing a differentiability point of every Lipschitz function}, Math. Ann. 351(3) (2011), 633--663.
\bibitem{DM12} Dor\'e, M., Maleva, O.: \emph{A compact universal differentiability set with Hausdorff dimension one}, Israel J. Math. 191(2) (2012), 889--900.
\bibitem{DMf} Dor\'e, M., Maleva, O.: \emph{A universal differentiability set in Banach spaces with separable dual}, J. Funct. Anal. 261 (2011), 1674--1710.
\bibitem{DM14} Dymond, M., Maleva, O.: \emph{Differentiability inside sets with upper Minkowski dimension one}, Michigan Math. J. 65 (2016), 613--636.
\bibitem{Fit84} Fitzpatrick, S.: \emph{Differentiation of real-valued functions and continuity of metric projections}, Proc. Amer. Math. Soc. 91(4) (1984), 544--548.
\bibitem{Gro96} Gromov, M.: \emph{Carnot-Carath\'{e}odory spaces seen from within}, Birkhauser Progress in Mathematics 144 (1996), 79--323.
\bibitem{HajMal} Hajlasz, P., Malekzadeh, S.: \emph{On conditions for unrectifiability of a metric space}, Anal. Geom. Metr. Spaces 3 (2015), 1--14.
\bibitem{LeD} Le Donne, E.: \emph{Lecture notes on sub-Riemannian geometry}, notes available at https://sites.google.com/site/enricoledonne/.
\bibitem{LS16} Le Donne, E., Speight, G.: \emph{Lusin Approximation for Horizontal Curves in Step 2 Carnot Groups}, Calc. Var. Partial Differential Equations, 55(5) (2016), 1--22.
\bibitem{LPT13} Lindenstrauss, J., Preiss, D., Tiser, J.: \emph{Fr\'echet differentiability of Lipschitz functions and porous sets in Banach spaces}, Annals of Mathematics Studies 179, Princeton University Press (2012).
\bibitem{Mag00} Magnani, V.: \emph{Unrectifiability and rigidity in stratified groups}, Arch. Math. (Basel) 83(6) (2004), 568--576. 
\bibitem{Mag01} Magnani, V.: \emph{Differentiability and area formula on stratified Lie groups}, Houston J. Math. 27(2) (2001), 297--323.
\bibitem{MR}  Magnani, V., Rajala, T.:  \emph{Radon-Nikodym property and area formula for Banach homogeneous group targets}, Int. Math. Res. Not. IMRN 23 (2014), 6399--6430.
\bibitem{MPS17} Magnani, V., Pinamonti, A., Speight, G.: \emph{Differentiability for Lipschitz maps from stratified groups to Banach homogeneous groups}, submitted, preprint available at arXiv: 1706.01782.
\bibitem{Mon02} Montgomery, R.: \emph{A tour of subriemannian geometries, their geodesics and applications}, American Mathematical Society, Mathematical Surveys and Monographs 91 (2006).
\bibitem{Marchi14} Marchi, M.: \emph{Regularity of sets with constant intrinsic normal in a class of Carnot groups}, Ann. Inst. Fourier (Grenoble), 64(2) (2014), 429--455.
\bibitem{Pan89} Pansu, P.: \emph{Metriques de Carnot-Carath\'{e}odory et quasiisometries des espaces symetriques de rang un}, Ann. of Math. (2) 129(1) (1989), 1--60. 
\bibitem{PS16} Pinamonti, A., Speight, G.: \emph{A Measure Zero Universal Differentiability Set in the Heisenberg Group}, Math. Ann. 368 (2017), no. 1-2, 233--278. 
\bibitem{PS16pansu} Pinamonti, A., Speight, G.: \emph{Porosity, Differentiability and Pansu's Theorem}, J. Geom. Anal. 27(3) (2017), 2055--2080.
\bibitem{PS16structure} Pinamonti, A., Speight, G.: \emph{Structure of Porous Sets in Carnot Groups}, to appear in Illinois J. Math., preprint available at arXiv:1607.04681.
\bibitem{Pre90} Preiss, D.: \emph{Differentiability of Lipschitz functions on Banach spaces}, J. Funct. Anal. 91(2) (1990), 312--345.
\bibitem{PS15} Preiss, D., Speight, G.: \emph{Differentiability of Lipschitz functions in Lebesgue null sets}, Inven. Math. 199(2) (2015), 517--559.
\bibitem{Sem} Semmes, S.: \emph{Quasisymmetry, measure and a question of Heinonen}, Rev. Mat. Iberoamericana 12(3) (1996), 727--781. 
\bibitem{Spe14} Speight, G.: \emph{Lusin Approximation and Horizontal Curves in Carnot Groups}, Rev. Mat. Iberoamericana 32(4) (2016), 1425--1446.
\bibitem{Vit14} Vittone, D.: \emph{The regularity problem for sub-Riemannian geodesics}, Geometric measure theory and real analysis, CRM Series, 17, Ed. Norm., Pisa (2014), 193--226.
\bibitem{Zah46} Zahorski, Z.: \emph{Sur l'ensemble des points de non-derivabilite d'une fonction continue}, Bull. Soc. Math. France 74 (1946), 147--178.
\end{thebibliography}
\end{document}